%% file: cgrowth-finalcorrected.tex
\documentclass[11pt]{amsart}
\usepackage{amsfonts, amsmath, amssymb, amscd, amsthm, graphicx,enumerate}
\usepackage{epsfig, color}
\usepackage{hyperref}

\makeatletter
\def\blfootnote{\xdef\@thefnmark{}\@footnotetext}
\makeatother

\newtheorem{thm}{Theorem}[section]
\newtheorem{cor}[thm]{Corollary}
\newtheorem{lem}[thm]{Lemma}

\theoremstyle{definition}
\newtheorem{defn}[thm]{Definition}
\theoremstyle{remark}
\newtheorem{rem}[thm]{Remark}
\newtheorem{ex}[thm]{Example}

\newfont{\eufm}{eufm10}

\newcommand{\Hl}{\{ H_\lambda \} _{\lambda \in \Lambda } }
\newcommand{\e}{\varepsilon }
\renewcommand{\phi}{\varphi}

\newcommand{\lab}{{\bf Lab}}
\newcommand{\Ker}{{\rm Ker\,}}
\newcommand{\Gs}{\Gamma (G, S)}
\renewcommand{\l}{{\ell}}
\newcommand{\Ga }{\Gamma (G, \mathcal A)}
\newcommand{\N}{\mathbb N}
\newcommand{\Z}{\mathbb Z}

\renewcommand{\ll }{\langle\hspace{-.7mm}\langle }
\newcommand{\rr }{\rangle\hspace{-.7mm}\rangle }

\newcommand{\h}{\hookrightarrow _h }
\renewcommand{\d }{{\rm dist}}
\begin{document}

\title{Conjugacy growth of finitely generated groups}
\author{Michael Hull}
\address[Michael Hull]{Department of Mathematics, Vanderbilt University, Nashville TN 37240, USA.}
\email{michael.b.hull@vanderbilt.edu}
\author{Denis Osin}
\address[Denis Osin]{Department of Mathematics, Vanderbilt University, Nashville TN 37240, USA.}
\email{denis.osin@gmail.com}

\thanks{The research was supported by the NSF grant DMS-1006345. The second author was also supported by the RFBR grant 11-01-00945}
\date{}

\begin{abstract}
The conjugacy growth function of a finitely generated group measures the number of conjugacy classes in balls with respect to a word metric. We study the following natural question: Which functions can occur as conjugacy growth function of finitely generated groups? Our main result answers the question completely. Namely we prove that a function $f\colon \mathbb N\to \mathbb N$ can be realized (up to a natural equivalence) as the conjugacy growth function of a finitely generated group if and only if $f$ is non-decreasing and bounded from above by $a^n$ for some $a\ge 1$. We also construct a finitely generated group $G$ and a subgroup $H\le G$ of index $2$ such that $H$ has only $2$ conjugacy classes while the conjugacy growth of $G$ is exponential. In particular, conjugacy growth is not a quasi-isometry invariant.
\end{abstract}


\maketitle


\section{Introduction}
Let $G$ be a group generated by a set $X$. Recall that the {\it word length} of an element $g\in G$ with respect to the generating set $X$, denoted by $|g|_X$, is the length of a shortest word in $X\cup X^{-1}$ representing $g$ in the group $G$. If $X$ is finite one can consider the {\it growth function} of $G$, $\gamma _G\colon\N\to \N $, defined by $$\gamma _G(n)=|B_{G,X}(n)|,$$ where $$B_{G,X}(n)=\{ g\in G\mid |g|_X\le n\}. $$ It was first introduced by Efremovic \cite{E} and Svar\v c \cite{S} in the 50's, rediscovered by Milnor \cite{Mil1} in the 60's, and served as the starting point and a source of motivating examples for contemporary geometric group theory. In this paper we focus on a similar function $\xi _{G,X}\colon\N\to \N $ called the {\it conjugacy growth function} of $G$ with respect to $X$. By definition $\xi _{G,X}(n)$ is the number of conjugacy classes in the ball $B_{G,X}(n)$.

It is straightforward to verify that $\gamma_{G,X}$ and $\xi _{G,X}$ are independent of the choice of a particular finite generating set $X$ of $G$ up to the following equivalence relation. Given $f,g\colon\N \to \N $, we write $f\preceq g$ if there exists $C\in \mathbb N$ such that  $f(n)\le g(Cn)$ for all $n\in \N$. Further $f$ and $g$ are {\it equivalent} (we write $f\sim g$) if  $ f\preceq g$ and $g\preceq f$. In what follows we always consider growth functions up to this equivalence relation and omit $X$ from the notation.

The conjugacy growth function was introduced by Babenko \cite{IB} in order to study geodesic growth of Riemannian manifolds. Obviously free homotopy classes of loops in a manifold $M$  are in $1$-to-$1$ correspondence with conjugacy classes of $\pi _1(M)$. If $M$ is a closed Riemannian manifold, the proposition known as the Svar\v c--Milnor Lemma (and first proved by Efremovic in \cite{E}) then implies that $\xi _{\pi _1(M)}$ is equivalent to the function counting free homotopy classes of loops of given length in $M$. The later function serves as a lower bound for the geodesic growth function of $M$, which counts the number of geometrically distinct closed geodesics of given length on $M$. Moreover if $M$ has negative sectional curvature, then all these functions are equivalent.

Geodesic growth of compact Riemannian manifolds has been studied extensively since late 60's (see, e.g., \cite{BTZ,BaHi,Kat,M}). The most successful results were obtained in the case of negatively curved manifolds by Margulis \cite{M,MS}. He proved that the number of primitive closed geodesics of length at most $n$ on a closed manifold of negative sectional curvature is approximately equal to ${e^{hn}}/(hn)$, where $h$ is the topological entropy of the geodesic flow on the unit tangent bundle of the manifold. Coornaert and Knieper \cite{CK1,CK2} proved a group theoretic analogue of this result and found an asymptotic estimate for the number of primitive conjugacy classes in a hyperbolic group similar to that from Margulis' papers.

Recall that a conjugacy class of a group $G$ is called {\it primitive} if some (or, equivalently, any) element $g$ from the class is not a proper power, i.e., $h^n=g$ implies $n=\pm 1$. For a group $G$ generated by a finite set $X$, let $\pi _G(n)$ denote the function counting primitive conjugacy classes in $B_{G,X}(n)$. It is not hard to show that  $\pi _G$ and $\xi _G$ are equivalent and grow exponentially for many `hyperbolic-like' groups. Indeed we prove the following.

\begin{thm}[Theorem \ref{main0}]\label{intromain0}
Let $G$ be a finitely generated group with a non-degenerate hyperbolically embedded subgroup. Then $\xi _G\sim \pi _G\sim 2^n$.
\end{thm}

The notion of a hyperbolically embedded subgroup was introduced in \cite{DGO}. The condition that the subgroup is non-degenerate simply means that it is proper and infinite. Groups containing non-degenerate hyperbolically embedded subgroups include non-elementary hyperbolic and, more generally, relatively hyperbolic groups with proper parabolic subgroups, all but finitely many mapping class groups of closed orientable surfaces (possibly with punctures), $Out(F_n)$ for $n\ge 2$, the Cremona group ${\rm Bir} ({\mathbb P^2_{\mathbb C}})$ (i.e., the group of birational automorphism of the complex projective plane), and many other examples \cite{DGO}. Theorem \ref{intromain0} can be used to completely classify conjugacy growth functions of subgroups in mapping class groups (see Section 3).

On the other hand, counting primitive conjugacy classes does not make much sense for general groups. For instance if the group $G$ is torsion without involutions, then there are no primitive conjugacy classes in $G$ at all.  Moreover this can happen even for torsion free groups. The simplest (but not finitely generated) example is the group $\mathbb Q$. Finitely generated examples of torsion free groups $G$ where every element $g\in G$ is a proper power (and, moreover, for every $n\in \mathbb Z\setminus \{ 0\} $ the equation $x^n=g$ has a solution in $G$) were first constructed by Guba in  \cite{Gub}. Thus in the context of abstract group theory it seems more natural to consider the conjugacy growth function $\xi _G$.

The algebraic study of the conjugacy growth function is strongly motivated by its similarity to the ordinary growth function. Recall that a function $f$ is {\it exponential} if $f\sim 2^n$, {\it polynomial} if $f\sim n^d$ for some $d\in \mathbb N$, and {\it polynomially bounded} if $f\preceq n^d$ for some $d\in \mathbb N$. In \cite{Mil1}, Milnor conjectured that $\gamma_G$ is always either exponential or polynomial. Counterexamples to this conjecture were constructed by Grigorchuk in \cite{Gri}. It turns out, however, that Milnor's dichotomy does hold for some important classes of groups including solvable and linear ones \cite{Mil2,T,W}. Gromov \cite{Gr} proved that any group with polynomially bounded growth function contains a nilpotent subgroup of finite index. Combining this with a result of  Bass \cite{B} saying that every nilpotent group has a polynomial growth function, one can easily derive that if the growth function of a group is polynomially bounded, then it is in fact polynomial. Despite these advances, it is still far from being clear which functions can occur as growth functions of finitely generated groups. For instance it is unknown whether there exists a group $G$ with non-polynomial growth function satisfying $\gamma _G \preceq 2^{\sqrt{n}}$.  For a comprehensive survey we refer the interested reader to \cite{Gri91}.

Recently some similar results were proved for the conjugacy growth function. Breuillard and Cornulier \cite{BC} showed that for a finitely generated solvable group $G$, the conjugacy growth function is either polynomially bounded or exponential and, furthermore, $\xi _G$ is polynomially bounded if and only if $G$ is virtually nilpotent. (For polycyclic groups this result was proved independently and simultaneously by the first author in \cite{H}.) This dichotomy was also proved for finitely generated linear groups by Breuillard, Cornulier, Lubotzky, and Meiri  \cite{BCLM}. Motivated by the Milnor conjecture, Guba and Sapir \cite{GubS} suggested that `natural' groups have either polynomially bounded or exponential conjugacy growth. They proved that many $HNN$-extensions and diagram groups, including the R. Thompson group $F$, have exponential conjugacy growth.

Note however that the ordinary and conjugacy growth functions can behave differently. For instance, the conjugacy growth of a nilpotent group is not necessary polynomial. Indeed let $H$ be the Heisenberg group
\[
H={\rm UT}_3(\mathbb Z)\cong \langle a,b,c \; | \; [a,b]=c, [a,c]=[b,c]=1 \rangle .
\]
Then it is fairly easy to compute that $\xi_H(n)\sim n^2\log(n)$ (this example can be found in \cite{IB} and \cite{GubS}). Note also that there exist finitely generated groups of exponential growth with finitely many conjugacy classes \cite[Theorem 41.2]{Ols-book} and even with $2$ conjugacy classes \cite{Osi10}. Thus $\gamma _G$ and $\xi_G$ can be very far apart, actually on the opposite sides of the spectrum.

The main goal of this paper is to address the following realization problem: \textit{Which functions can be realized (up to equivalence) as conjugacy growth functions of finitely generated groups}? Unlike in the case of ordinary growth, the realization problem for conjugacy growth admits a complete solution.

\begin{thm}[Theorem \ref{main}]\label{intromain}
Let $G$ be a group generated by a finite set $X$, $f$ the conjugacy growth function of $G$ with respect to $X$. Then the following conditions hold.
\begin{enumerate}
\item[(a)] $f$ is non-decreasing.
\item[(b)] There exists $a\ge 1$ such that $f(n) \le  a^n$ for every $n\in\N$.
\end{enumerate}
Conversely, suppose that a function $f\colon\N\to \N$ satisfies the above conditions (a) and (b). Then there exists an infinite finitely generated group $G$ such that $\xi_G\sim f$.
\end{thm}

The first claim of the theorem is essentially trivial. Note, however,  that even realizing simplest growth functions, e.g., $f(n)=\log n$, is nontrivial; moreover, we are not aware of any groups other than the ones constructed in this paper that have unbounded conjugacy growth functions satisfying $f(n)=o(n)$.

When speaking about asymptotic invariants of groups it is customary to ask whether these invariants are {\it geometric}, i.e. invariant under quasi-isometry. (Recall that quasi-isometry is a coarse analogue of the notion of isometry between metric spaces; for details and motivation we refer to \cite{Gr3}). Many asymptotic invariants of groups are invariant under quasi-isometry up to suitable equivalence relations, e.g., the ordinary growth function, the Dehn function, and the asymptotic dimension growth function, just to name a few. However it turns out that the conjugacy growth function is not a geometric invariant in the strongest possible sense. More precisely, we construct the following example.

\begin{thm}[Theorem\ref{main2}]\label{intromain2}
There exists a finitely generated group $G$ and a finite index subgroup $H\le G$ such that $H$ has $2$ conjugacy classes while $G$ is of exponential conjugacy growth.
\end{thm}

Since every (non-trivial) group has at least two conjugacy classes and at most exponential conjugacy growth, this theorem shows that the conjugacy growth of two quasi-isometric groups can be as far apart as possible. Note also that it is fairly easy to prove that for every finitely generated group $G$ and a finite index subgroup $H\le G$, one has $\xi _H\preceq \xi_G$.

The proofs of Theorems \ref{intromain} and \ref{intromain2} will be accomplished by constructing groups which are direct limits of relatively hyperbolic groups. The main tool in this procedure will be the theory of small cancellation over relatively hyperbolic groups. The idea of generalizing small cancellation theory to groups acting on hyperbolic spaces goes back to Gromov's paper \cite{Gr2}. In the case of hyperbolic groups, it was formalized by Champetier \cite{Ch}, Delzant \cite{Del}, Olshanskii \cite{Ols}, and others. Olshanskii's approach was generalized to relatively hyperbolic groups by the second author in \cite{Osi10} and this generalization will be used in our paper. Some difficulties occurring  in the proof of Theorem \ref{intromain} are discussed in Section 6.  To overcome these difficulties, we obtain new results about conjugate elements and elementary subgroups in small cancellation quotients of relatively hyperbolic groups, which can be useful elsewhere.

The paper is organized as follows. In Section 2 we collect necessary definitions and facts about relatively hyperbolic groups. Section 3 contains the proof of Theorem \ref{intromain0} and some related results. Sections 4 and 5 focus on small cancellation quotients of relatively hyperbolic groups and provide us with necessary tools for the rest of the paper. Sections 6 and 7 contain the proofs of Theorem \ref{intromain} and Theorem \ref{intromain2}, respectively.

\section{Preliminaries}


\paragraph{\bf Notation} We begin by standardizing the notation that will be used for the remainder of the paper. Given a group $G$ generated by a subset $S\subseteq G$, we denote by $\Gs $ the Cayley graph of $G$ with respect to $S$. That is, $\Gs$ is the graph with vertex set $G$ and an edge labeled by $s$ between each pair of vertices of the form $(g, gs)$, where $s\in S$. We will assume all generating sets are symmetric, that is $S=S\cup S^{-1}$. If $p$ is a (combinatorial) path in $\Gs$, $\lab (p)$ denotes its label, $\l(p)$ denotes its length, $p_-$ and $p_+$ denote its starting and ending vertex.

For a word $W$ in an alphabet $S$, $\|W\|$ denotes its length. For two words $U$ and $V$ we write $U \equiv V$ to denote the letter-by-letter equality between them. Clearly there is a one to one correspondence between words $W$ and paths $p$ such that $p_-=1$ and $\lab(p)\equiv W$.

The normal closure of a subset $K\subseteq G$ in a group $G$ (i.e., the minimal normal subgroup of $G$ containing $K$) is denoted by $\ll K\rr^G$, or simply by $\ll K\rr$ if omitting $G$ does not lead to a confusion.  For  group elements $g$ and $t$, $g^t$ denotes $t^{-1}gt$.  We write $g\sim h$ if $g$ is conjugate to $h$, that is  there exists $t\in G$ such that $g^t=h$.

\paragraph{\bf Van Kampen Diagrams.}
Recall that a {\it van Kampen diagram} $\Delta $ over a presentation
\begin{equation}
G=\langle \mathcal A\; | \; \mathcal O\rangle \label{ZP}
\end{equation}
is a finite, oriented, connected, simply--connected, planar 2--complex endowed with a labeling
function $\lab\colon E(\Delta )\to \mathcal A$, where $E(\Delta ) $
denotes the set of oriented edges of $\Delta $, such that $\lab
(e^{-1})\equiv (\lab (e))^{-1}$. Labels and lengths of paths are defined as in the case of Cayley graphs. Given a cell $\Pi $ of $\Delta $, we denote by $\partial \Pi$ the boundary of $\Pi $; similarly, $\partial \Delta $ denotes the boundary of $\Delta $. The labels of $\partial \Pi $ and $\partial \Delta $ are defined up to a cyclic permutation. An additional requirement is that for any cell $\Pi $ of $\Delta $, the boundary label $\lab (\partial
\Pi)$ is equal to a cyclic permutation of a word $P^{\pm 1}$, where $P\in \mathcal O$. The van Kampen Lemma states that a word $W$ over the alphabet
$\mathcal A$ represents the identity in the group given by
(\ref{ZP}) if and only if there exists a
diagram $\Delta $ over (\ref{ZP}) such that $\lab(\partial \Delta
)\equiv W$ \cite[Ch. 5, Theorem 1.1]{LS}.

\begin{rem}\label{cm}
For every van Kampen diagram $\Delta $ over (\ref{ZP}) and any
fixed vertex $o$ of $\Delta $, there is a (unique) combinatorial
map $\gamma \colon Sk^{(1)} (\Delta )\to \Ga $ (where $Sk^{(1)} (\Delta )$ denotes the 1-skeleton of $\Delta$)  that preserves labels and orientation of edges and maps $o$ to the vertex $1$ of
$\Ga $.
\end{rem}
\paragraph{\bf Relatively hyperbolic groups.} The notion of a relatively hyperbolic group was originally suggested by Gromov in \cite{Gr2}. In \cite{Bo} this idea was elaborated on by Bowditch, who suggested a definition in terms of the dynamics of properly discontinuous
isometric group actions on hyperbolic spaces. Alternatively, another definition was suggested by Farb in \cite{F} who looked at the hyperbolicity of a certain graph associated to a group and a collection of subgroups, called the {\it coset graph}. These definitions are not equivalent, but Farb also considered another property called {\it bounded coset penetration}, or BCP. It turns out that being relatively hyperbolic in the sense of Bowditch is equivalent to being relatively hyperbolic in the sense of Farb and satisfying BCP (see, for example \cite{Bo}). In \cite{Osi06a}, the second author gave an isoperimetric characterization of relative hyperbolicity which generalizes these definitions to the case of non-finitely generated groups. We will present the definition found there and refer the reader to \cite{Hru,Osi06a} for more details.

\begin{defn}\label{relpresdef}Let $G$ be a group, $\Hl $ a collection of
subgroups of $G$, $X$ a subset of $G$. We say that $X$ is a {\it
relative generating set of $G$ with respect to $\Hl $} if $G$ is
generated by $X$ together with the union of all $H_\lambda $. In this situation the group
$G$ can be regarded as a quotient group of the free product
\begin{equation}
F=\left( \ast _{\lambda\in \Lambda } H_\lambda  \right) \ast F(X),
\label{F}
\end{equation}
where $F(X)$ is the free group with the basis $X$. Let $N$ denote
the kernel of the natural homomorphism $F\to G$. If $N$ is the
normal closure of a subset $\mathcal Q\subseteq N$ in the group
$F$, we say that $G$ has {\it relative presentation}
\begin{equation}\label{G}
\langle X,\; H_\lambda, \lambda\in \Lambda \; |\; \mathcal Q
\rangle .
\end{equation}
If $|X|<\infty $ and $|\mathcal Q|<\infty $, the
relative presentation (\ref{G}) is said to be {\it finite} and the
group $G$ is said to be {\it finitely presented relative to the
collection of subgroups $\Hl $.}
\end{defn}

Set
\begin{equation}\label{H}
\mathcal H=\bigsqcup\limits_{\lambda\in \Lambda} (H_\lambda
\setminus \{ 1\} ) .
\end{equation}
Given a word $W$ in the alphabet $X\cup \mathcal H$ such that $W$
represents $1$ in $G$, there exists an expression
\begin{equation}
W=_F\prod\limits_{i=1}^k f_i^{-1}Q_i^{\pm 1}f_i \label{prod}
\end{equation}
with the equality in the group $F$, where $Q_i\in \mathcal Q$ and
$f_i\in F $ for $i=1, \ldots , k$. The smallest possible number
$k$ in a representation of the form (\ref{prod}) is called the
{\it relative area} of $W$ and is denoted by $Area^{rel}(W)$.

\begin{defn}
A group $G$ is {\it hyperbolic relative to a collection of
subgroups} $\Hl $ if $G$ is finitely presented relative to $\Hl $
and there is a constant $L>0$ such that for any word $W$ in $X\cup
\mathcal H$ representing the identity in $G$, we have $Area^{rel}
(W)\le L\| W\| $.
\end{defn}

If $G$ is hyperbolic relative to $\Hl$, then the subgroups  in $ \Hl$ are called {\it parabolic subgroups}. Observe also that the relative area of a word $W$ representing $1$
in $G$ can be defined geometrically via van Kampen diagrams. Let
$G$ be a group given by the relative presentation (\ref{G}) with
respect to a collection of subgroups $\Hl $.  We denote by
$\mathcal S$ the set of all words in the alphabet $\mathcal H$
representing the identity in the group $F$ defined by (\ref{F}).
Then $G$ has the ordinary (non--relative) presentation
\begin{equation}\label{Gfull}
G=\langle X\cup\mathcal H\; |\;\mathcal S\cup \mathcal Q \rangle .
\end{equation}
A cell in van Kampen diagram $\Delta $ over (\ref{Gfull}) is
called a {\it $\mathcal Q$--cell} if its boundary is labeled by a
word from $\mathcal Q$. We denote by $N_\mathcal Q(\Delta )$ the
number of $\mathcal Q$--cells of $\Delta $. Obviously given a word
$W$ in $X\cup\mathcal H$ that represents $1$ in $G$, we have
$$
Area^{rel}(W)=\min\limits_{\lab (\partial \Delta ) \equiv W} \{
N_\mathcal Q (\Delta )\} ,
$$
where the minimum is taken over all van Kampen diagrams with
boundary label $W$. Thus, a group $G$ is hyperbolic relative to $\Hl$ if it is finitely presented with respect to $\Hl$ and all van Kampen diagrams over (\ref{Gfull}) satisfy a linear relative isoperimetric inequality.

In particular, a group $G$ is an ordinary (Gromov) {\it hyperbolic group} if $G$ is hyperbolic relative to the trivial subgroup. Another example (which is important for our purposes) is that any free product of groups is hyperbolic relative to the factors, since in this case $X=\mathcal Q=\emptyset$. More examples of relatively hyperbolic groups can be found in \cite{Osi07}.

The following useful result can be easily derived from the definition.

\begin{lem}[\cite{Osi06a}, Theorem 1.4]\label{malnorm}
Let $G$ be a group hyperbolic relative to a collection of subgroups $\Hl $. Then for every $\lambda \in \Lambda $ and $g\in G\setminus H_\lambda $, we have $|H_\lambda \cap H_\lambda ^g|<\infty $. Also, if $\mu\neq\lambda$, then $|H_\mu \cap H_\lambda ^g|<\infty$ for all $g\in G$.  In particular, if $G$ is torsion free, then every parabolic subgroup is malnormal, and any two elements in distinct parabolic subgroups are non-conjugate.
\end{lem}

\cite{DS} gives a characterization of the asymptotic cones of relatively hyperbolic groups, and an immediate consequence of this is the following.

\begin{lem}[\cite{DS}, Corollary 1.14]\label{relsub}
If a group $G$ is hyperbolic relative to $\{H_1,..., H_m\}$, and each $H_i$ is hyperbolic relative to a collection of subgroups $\{H_1^i,..., H_{n_i}^i\}$, then G is hyperbolic relative to $\{H_j^i \;|\; 1\leq i\leq m, 1\leq j\leq n_i\}$.
\end{lem}

The next lemma is a particular case of \cite[Theorem 2.40]{Osi06a}.
\begin{lem}\label{exhyp}
Suppose that a group $G$ is hyperbolic relative to a collection of
subgroups $\Hl \cup \{ S_1, \ldots , S_m\} $, where $S_1, \ldots ,
S_m $ are finitely generated and hyperbolic in the ordinary
(non--relative) sense. Then $G$ is hyperbolic relative to $\Hl $.
\end{lem}

Recall that a metric space $M$ is {\it $\delta $--hyperbolic} for some $\delta \ge 0$ (or simply {\it hyperbolic}) if for any geodesic triangle $T$ in $M$, any side of $T$ belongs to the union of the closed $\delta $--neighborhoods of the other two sides. As mentioned above, $G$ is an ordinary hyperbolic group if $G$ is
hyperbolic relative to the trivial subgroup. An equivalent
definition says that $G$ is hyperbolic if it is generated by a
finite set $X$ and the Cayley graph $\Gamma (G, X)$ is a hyperbolic metric space.
In the relative case these approaches are not equivalent, but we
still have the following, which will provide one of the main tools for looking at small cancellation quotients of relatively hyperbolic groups.

\begin{thm}[\cite{Osi06a}, Theorem 1.7]\label{hms}
Let $G$ be a group hyperbolic relative to a collection of subgroups $\Hl $, $X$ a finite relative generating set for $G$. Then the Cayley graph $\Gamma(G, X\cup \mathcal H)$ is a hyperbolic metric space.
\end{thm}

This theorem will allow us to apply the following useful lemma, which appears in \cite[Corollary 3.3]{Osi10} and can be derived from basic properties of hyperbolic spaces (see, for example, \cite{BH}). A path $p$ in a metric space is called {\it $(\lambda, c)$--quasi--geodesic} for some $\lambda > 0$, $c\ge 0$,  if
$$\d(q_-, q_+)\ge \lambda l(q)-c$$ for any subpath $q$ of $p$.

\begin{lem}\label{qgq}
For any $\delta \ge 0$, $\lambda > 0$, $c\ge 0$, there exists a
constant $K=K(\delta , \lambda , c)$ with the following property.
Let $Q$ be a quadrangle in a $\delta$--hyperbolic space whose
sides are $(\lambda , c)$--quasi--geodesic. Then each side of $Q$
belongs to the closed $K$--neighborhood of the
union of the other three sides.
\end{lem}

\paragraph{\bf Loxodromic elements and elementary subgroups.}
We call an element $g\in G$ {\it parabolic} if it is conjugate to an element of one of the
parabolic subgroups. A non-parabolic elements of infinite order are called {\it loxodromic}. Broadly speaking, most algebraic properties of elements in hyperbolic groups also hold for loxodromic elements of relatively hyperbolic groups. An example of this is the following lemma; recall that a group is {\it elementary} if it contains a cyclic subgroup of finite index.

\begin{lem}[\cite{Osi06b}] \label{Eg}
Suppose a group $G$ is hyperbolic relative to a collection of subgroups $\Hl$. Let $g$ be a
loxodromic element of $G$. Then the following conditions hold:
\begin{enumerate}
\item[(a)] There is a unique maximal elementary
subgroup $E_G(g)\le G$ containing $g$.

\item[(b)] $E_G(g)=\{ h\in G\mid \exists\, m\in \mathbb{N}~ \mbox{such that}~ h^{-1}g^mh=g^{\pm m}\} $.

\item[(c)] The group $G$ is hyperbolic relative to the collection
$\Hl\cup \{ E_G(g)\} $.
\end{enumerate}
\end{lem}

The following is an immediate consequence of Lemma \ref{Eg}. Recall that elements $f, g\in G$ are called {\it commensurable} if there exist $k, l\in \Z\setminus\{0\}$ such that $f^k\sim g^l$.
\begin{lem}\label{primcom}
Let $f$ and $g$ be primitive loxodromic elements in a torsion free relatively hyperbolic group. Then $f$ is commensurable with $g$ if and only if $f^{\pm 1}\sim g$.
\end{lem}
\begin{proof}
If $f^k=(g^l)^x$, then $\langle f\rangle=E_G(f)=E_G(g^x)=\langle g^x\rangle$, thus $f^{\pm 1}= g^x$.
\end{proof}

Since any group will be hyperbolic relative to itself, we will need some non-trivial structure outside of the parabolic subgroups. This will be accomplished with the notion of suitable subgroups.
\begin{defn} \label{suit} A subgroup $S\le G$ is called
{\it suitable} if there exist two non--commensurable loxodromic elements
$s_1, s_2\in S$ such that $E_G(s_1)\cap E_G(s_2)=1$.
\end{defn}

The next lemma is a combination of  \cite[Lemma 2.3]{Osi10} and \cite[Proposition 3.4]{AMO}.
\begin{lem}\label{non-com}
Let $G$ be a group hyperbolic relative to a collection of
subgroups $\Hl $.
\begin{enumerate}
\item
 If $S$ is a suitable subgroup of $G$, then there exist
infinitely many pairwise non--commensurable loxodromic elements $s_1,
s_2, \ldots \in S$ such that for all $i=1,2, \ldots $,
$E_G(s_i)=\langle s_i\rangle $. In particular, $E_G(s_i)\cap
E_G(s_j)=\{ 1\} $ whenever $i\ne j$.
\item
If $G$ is torsion free, then any non-elementary subgroup containing at least one loxodromic element is suitable.
\end{enumerate}
\end{lem}

\paragraph{\bf HNN-extensions and relative hyperbolicity.}
Given a group $G$ containing two isomorphic subgroups $A$ and $B$, the HNN-extension $G\ast_{A^t=B}$ is the group given by
\[
G\ast_{A^t=B}=\langle G,t\; |\; t^{-1}at=\phi(a),\; a\in A\rangle
\]
where $\phi\colon A\to B$ is an isomorphism. Recall that for a word $W$ in the alphabet $\{G\setminus\{1\}, t\}$, a {\it pinch} (in the HNN-extension $G\ast_{A^t=B}$) is a subword of the form $t^{-1}at$ with $a\in A$ or $tbt^{-1}$ with $b\in B$.  a word $g_0t^{\e_0}g_1t^{\e_1}...g_{n-1}t^{\e_{n-1}}g_n$, where each $g_i\in G$ and each $\e_i=\pm 1$, is called {\it reduced} if there are no pinches. Given such a word, we define its  {\it t-length} as the number of occurrences of the letters $t$ and $t^{-1}$. In $G\ast_{A^t=B}$, any pinch can be replaced by a single element of $G$. It follows that each element $w\in G\ast_{A^t=B}$ is equal to a reduced word. The converse to this statement is known as the Britton Lemma (see \cite[Ch. 4, Sec.2]{LS}).

\begin{lem}[Britton Lemma]\label{BL}
Let  $W$ be a word in $\{G\setminus\{1\}, t\}$ with $t$-length at least $1$ and no pinches. Then $W\neq 1$ in $G\ast_{A^t=B}$.
\end{lem}

An immediate consequence of this lemma is the well-known fact that $G$ naturally embeds in the HNN-extension $G\ast_{A^t=B}$. When dealing with HNN-extensions of relatively hyperbolic groups we will often use the following corollary of the Britton Lemma. A reduced word $W$ is called {\it cyclicly reduced} if it is not conjugate to an element of shorter $t$-length, or equivalently, no cyclic shift contains a pinch.
\begin{lem}\label{conjHNN}
Let $G$ be a group, $A,B$ isomorphic subgroups of $G$. Suppose that some $f\in G $ is not conjugate to any elements of $A\cup B$ in $G$. Then in the corresponding HNN-extension  $G\ast_{A^t=B}$,
\begin{enumerate}
\item
$f$ is conjugate to another element $g\in G$ in $G\ast_{A^t=B}$ if and only if $f$ and $g$ are conjugate in $G$.
\item
If $f$ is primitive in $G$, then $f$ is primitive in $G\ast_{A^t=B}$.
\end{enumerate}
\end{lem}

\begin{proof}
Since $f$ is not in $A$ or $B$, if $W$ is any reduced word then $W^{-1}fWg^{-1}$ contains no pinches. Thus, $f$ is not conjugate to $g$. The second assertion will immediately follow if we can show that if $w\in G\ast_{A^t=B}$ such that $w^n\in G$, then either $w\in G$ or $w^n$ is conjugate to an element of $A$ or an element of $B$ (here we identify $G$ with its image in $G\ast_{A^t=B}$). To show this we induct on the $t$-length of the reduced form of $w$. If a reduced word representing $w$ contains no $t$ letters, then $w\in G$. Clearly $w^n\in G$ implies that $w$ has even $t$-length, since the sum of the exponents of $t$ letters must be $0$. Suppose $w$ has $t$-length $2$. Then for some $g_0, g_1, g_2$, $\e\in\{0,1\}$, we have that $w=g_0t^{\e}g_1t^{-\e}g_2$. The Britton Lemma implies that $t^{-\e}g_2g_0t^{\e}$ must be a pinch or freely trivial, that is $t^{-\e}g_2g_0t^{\e}=h$ for some (possibly trivial) $h$ in $A$ or $B$. Without loss of generality, let $h\in A$, and note that this implies that $g_2g_0\in B$. Then $w^n=g_0t^{\e}(g_1hg_1)^nt^{-\e}g_2$. Again, by the Britton Lemma $t^{\e}(g_1hg_1)^nt^{-\e}$ must be a pinch or trivial, and since the orientation of the $t$ is reversed, we get that $t^{\e}(g_1hg_1)^nt^{-\e}=h'$ for some $h'\in B$. Finally, observe that $w^n=g_0h'g_2g_0g_0^{-1}$, thus $w^n\sim h'g_2g_0\in B$.

Now suppose we have shown the above claim for all elements of $G\ast_{A^t=B}$ with shorter $t$-length then $w$. As before, $w=g_0t^{\e_0}...t^{-\e_0}g_n$, and $t^{-\e_0}g_ng_0t^{\e_0}=h$, for some $h\in A\cup B$. Let $u'=g_1t^{\e_1}...g_{n-1}h$. Now let $u$ be a conjugate of $u'$ which is cyclicly reduced. Since $u\sim u'=t^{-\e_0}g_0^{-1}wg_0t^{\e_0}$, we have that $u\sim w$ and so $u^n\sim w^n\in G$. Since $u$ is cyclicly reduced, $u^n$ is cyclicly reduced, hence $u^n\in G$. Since $u$ has fewer $t$ letters then $w$, by the inductive hypothesis, $u^n$ (and thus $w^n$) is conjugate to an element of $A$ or an element of $B$.
\end{proof}

The following result was first proved by Dahmani
in \cite{Dah} for finitely generated groups and then in \cite{Osi06c} in the full generality. It is worth noting that we will use it for infinitely generated groups in this paper.

\begin{lem}\label{HNN}
Suppose that a group $G$ is hyperbolic relative to a collection of
subgroups $\Hl \cup \{ K\} $ and for some $\nu \in \Lambda $,
there exists a monomorphism $\iota \colon K\to H_{\nu }$. Then the
HNN-extension
\begin{equation}\label{HNN-pres}
\langle G,t\; |\; t^{-1}kt=\iota (k),\; k\in K\rangle
\end{equation}
is hyperbolic relative to $\Hl $.
\end{lem}

\begin{cor}\label{sgHNN}
Let $G$ be a torsion free group hyperbolic relative to $\Hl$, $S$ a suitable subgroup of $G$, and $g$ a loxodromic element of $G$. Then for any $h\in\mathcal H$, there is an isomorphism $\iota\colon E_G(g)\to\langle h\rangle$ and the corresponding HNN-extension $G\ast_{E_G(g)^t=\langle h\rangle}$ is hyperbolic relative to $\Hl$. Furthermore, $S$ is a suitable subgroup of $G\ast_{E_G(g)^t=\langle h\rangle}$.
\end{cor}
\begin{proof}
The existence of $\iota$ follows from the fact that since $G$ is torsion free, $E_G(g)$ and $\langle h\rangle$ are both infinite cyclic. The relative hyperbolicity of $G\ast_{E_G(g)^t=\langle h\rangle}$ follows immediately from Lemma \ref{Eg} and Lemma \ref{HNN}. Since $S$ is suitable in $G$, Lemma \ref{non-com} yields the existence of infinitely many pairwise non--commensurable loxodromic (in $G$) elements of $S$.  At most one of these elements is commensurable with $g$ in $G$. Therefore, by Lemma \ref{conjHNN} $S$, considered as a subgroup of $G\ast_{E_G(g)^t=\langle h\rangle}$, will still contain loxodromic elements. Thus $S$ is suitable in $G\ast_{E_G(g)^t=\langle h\rangle}$ by Lemma \ref{non-com}.
\end{proof}


\section{Conjugacy growth in groups with hyperbolically embedded subgroups}


Let $G$ be a group, $H\le G$, $X\subseteq G$. We assume that $G=\langle X\cup H\rangle $ and denote by $\Gamma(G, X\cup H)$ the Cayley graph of $G$ with respect to the generating set $X\cup H$ and by $\Gamma _H$ the Cayley graph of $H$ with respect to the generating set $H$. Clearly $\Gamma _H$ is a complete subgraph of $\Gamma (G, X\cup H)$.

Given two elements $h_1,h_2\in H$, we define $\widehat d(h_1,h_2)$ to be the length of a shortest path $p$ in $\Gamma (G, X\cup H) $ that connects $h_1$ to $h_2$ and does not contain edges of $\Gamma _H$. If no such path exists we set $\widehat d(h_1, h_2)=\infty $. Clearly $\widehat d\colon H\times H\to [0, \infty]$ is a metric on $H$.

\begin{defn} We say that $H$ is {\it hyperbolically embedded in $G$ with respect to} $X\subseteq G$ (and write $H\hookrightarrow _h (G,X) $) if the following conditions hold:
\begin{enumerate}
\item[(a)] $G=\langle X\cup H\rangle$ and $\Gamma (G, X\cup H)$ is hyperbolic.
\item[(b)] $(H,\widehat d)$ is a locally finite metric space, i.e., every ball (of finite radius) is finite.
\end{enumerate}
We also say that $H$ is \textit{hyperbolically embedded in $G$} (and write $H\hookrightarrow _h G$) if $H\hookrightarrow _h (G,X) $ for some $X\subseteq G$.
\end{defn}

Note that for any group $G$, we have $G\h (G,X)$ for $X=\emptyset$. Indeed then the Cayley graph $\Gamma(G, X\cup H)$ has diameter $1$ and $d(h_1, h_2)=\infty $ whenever $h_1\ne h_2$. Further, if $H$ is a finite subgroup of a group $G$, then $H\h (G,X)$ for $X=G$. These cases are referred to as {\it degenerate}.

\begin{defn}
A hyperbolically embedded subgroup $H\h G$  is called \emph{non-degenerate} if it is infinite and proper (i.e., $H\ne G$).
\end{defn}

One may wonder if the case when $H$ is of finite index in $G$ should also be considered degenerate. In fact, a proper finite index subgroup of an infinite group is never hyperbolically embedded (see Lemma \ref{malnHE}).

Let us consider two elementary examples to illustrate the definition.

\begin{ex}
\begin{enumerate}
\item Let $G=H\times \mathbb Z$, $X=\{ x\} $, where $x$ is a generator of $\mathbb Z$. Then $\Gamma (G, X\cup H)$ is quasi-isometric to a line and hence it is hyperbolic. However $d(h_1, h_2)\le 3$ for every $h_1, h_2\in H$. If $H$ is infinite, then $H\not\h (G,X)$. Moreover, generalizing this argument, one can show that $H\not\h G$.

\item Let $G=H\ast \mathbb Z$, $X=\{ x\} $, where $x$ is a generator of $\mathbb Z$. In this case $\Gamma (G, X\cup H)$ is quasi-isometric to a tree and $d(h_1, h_2)=\infty $ unless $h_1=h_2$. Thus $H\h (G,X)$.
\end{enumerate}
\end{ex}

A group $G$ is hyperbolic relative to a subgroup $H$ if and only if $H\hookrightarrow _h (G,X) $ for some finite set $X$ \cite{DGO}. This provides us with a rich source of examples. For instance, the following groups contain non-degenerate hyperbolically embedded subgroups: non-elementary hyperbolic groups, fundamental groups of complete finite-volume manifolds of pinched negative sectional curvature, free products of groups other than $\mathbb Z_2\ast \mathbb Z_2$ and their small cancellation quotients as defined in \cite{LS}, groups with infinitely many ends, non-abelian finitely generated groups acting freely on $\mathbb R^n$-trees, including non-abelian limit groups, etc.

On the other hand, there are many examples of non-relatively hyperbolic groups that contain non-degenerate hyperbolically embedded subgroups. Examples include all but finitely many mapping class groups,   $Out(F_n)$ for $n\ge 2$, the Cremona group ${\rm Bir} ({\mathbb P^2_{\mathbb C}})$ (i.e., the group of birational automorphism of the complex projective plane), directly indecomposable right angled Artin groups, and many groups acting on trees. For details we refer to \cite{DGO}.

The following two results can be found in \cite{DGO}.

\begin{lem}\label{malnHE}
Let $G$ be a group, $H$ a hyperbolically embedded subgroup. Then for every $g\in G\setminus H$, the intersection $H\cap H^g$ is finite.
\end{lem}

\begin{lem}\label{Kn}
Let $G$ be a group with a non-degenerate hyperbolically embedded subgroup. Then for every $n\in \mathbb N$, there exists a hyperbolically embedded subgroup $K_n\le G$ such that $K_n\cong Z\times F_n$, where $Z$ is finite and $F_n$ is the free group of rank $n$.
\end{lem}

We are now ready to prove our first result.

\begin{thm}\label{main0}
Let $G$ be a finitely generated group with a non-degenerate hyperbolically embedded subgroup. Then $\xi _G\sim \pi _G\sim 2^n$.
\end{thm}

\begin{proof}
Let $K=K_2\le G$ be the subgroup provided by Lemma \ref{Kn}. We think of $F_2$ as a subgroup of $K$. Note first that if $g^m=f$ for some $g\in G$, $f\in K$, and $m\in \mathbb Z\setminus \{ 0\}$, then the intersection $K^g\cap K$ contains the subgroup $\langle f\rangle$. Hence if $f$ has infinite order, $g\in K$ by Lemma \ref{malnHE}. Thus every element of $K$ of infinite order that is primitive in $K$ is also primitive in $G$. Furthermore, if two elements of $K$ of infinite order are conjugate in $G$, then they are conjugate in $K$ for the same reason. Thus we obtain $\pi _G\succeq \pi _K$ and the later function is obviously exponential (this also follows from the results of \cite{CK1} as $K$ is non-elementary hyperbolic). Since $\pi _G\preceq \xi_G \preceq 2^n$ for every finitely generated group $G$, we are done.
\end{proof}

Theorem \ref{main0} can be used to completely classify conjugacy growth functions of subgroups of certain groups, e.g., mapping class groups.

\begin{cor}
Let $\Sigma $ be a (possibly punctured) closed orientable surface, $G$ a subgroup of the mapping class group of $\Sigma $. Then either $G$ is virtually abelian (in which case $\xi_G$ is polynomial), or $\xi_G$ is exponential.
\end{cor}

\begin{proof}
By Theorem 2.21 from \cite{DGO}, $G$ is either virtually abelian, or has a finite index subgroup $G_0$ which surjects on a group with a non-degenerate hyperbolically embedded subgroup. In the later case $\xi_{G_0}$ is exponential  by Theorem \ref{main0}. It is straightforward to prove that one has $\xi _{G_0}\preceq \xi_G$ whenever $[G:G_0]<\infty $ (see, e.g., \cite{H}). Hence the claim.
\end{proof}


\section{Small cancellation conditions}


Given a set of words $\mathcal R$ in an alphabet $\mathcal A$, we
say that $\mathcal R$ is {\it symmetrized} if for any $R\in
\mathcal R$, $\mathcal R$ contains all cyclic shifts of $R^{\pm
1}$. Further, if $G$ is a group generated by a set $\mathcal A$,
we say that a word $R$ is {\it $(\lambda , c)$--quasi--geodesic} in
$G$ if any path in the Cayley graph $\Gamma (G, \mathcal A)$
labeled by $R$ is $(\lambda , c)$--quasi--geodesic.

We begin by giving the small cancellation conditions introduced by Olshanskii in \cite{Ols} and also used in \cite{Osi10}.

\begin{defn}\label{piece}
Let $G$ be a group generated by a set $\mathcal A$, $\mathcal R$ a
symmetrized set of words in $\mathcal A$. For $\e
>0$, a subword $U$ of a word $R\in \mathcal R$ is called an {\it
$\e $--piece}  if there exists a word $R^\prime \in \mathcal R$
such that:

\begin{enumerate}
\item[(1)] $R\equiv UV$, $R^\prime \equiv U^\prime V^\prime $, for
some $V, U^\prime , V^\prime $; \item[(2)] $U^\prime = YUZ$ in $G$
for some words $Y,Z$ in $\mathcal A$ such that $\max \{ \| Y\| ,
\,\| Z\| \} \le \e $; \item[(3)] $YRY^{-1}\ne R^\prime $ in the
group $G$.
\end{enumerate}
Similarly, a subword $U$ of $R\in \mathcal R$ is called an {\it
$\e ^\prime $--piece}  if:
\begin{enumerate}
\item[($1^\prime $)] $R\equiv UVU^\prime V^\prime $ for some $V,
U^\prime , V^\prime $; \item[($2^\prime $)] $U^\prime =YU^{\pm
1}Z$ in the group $G$ for some $Y,Z$ satisfying $\max\{ \| Y\| ,
\| Z\| \}\le \e $.
\end{enumerate}
\end{defn}

\begin{defn}\label{DefSC}
We say that the set $\mathcal R$ satisfies the {\it $C(\e , \mu ,
\lambda , c, \rho )$--condition} for some $\e \ge 0$, $\mu >0$,
$\lambda >0$, $c\ge 0$, $\rho >0$, if
\begin{enumerate}
\item[(1)] $\| R\| \ge \rho $ for any $R\in \mathcal R$;
\item[(2)] any word $R\in \mathcal R$ is $(\lambda ,
c)$--quasi--geodesic; \item[(3)] for any $\e $--piece of any word
$R\in \mathcal R$, the inequality $\max \{ \| U\| ,\, \| U^\prime
\| \} < \mu \| R\| $ holds (using the notation of Definition
\ref{piece}).
\end{enumerate}
Further the set $\mathcal R$ satisfies the {\it $C_1(\e , \mu ,
\lambda , c, \rho )$--condition} if in addition the condition
$(3)$ holds for any $\e ^\prime $--piece of any word $R\in
\mathcal R$.
\end{defn}

Suppose that $G$ is a group defined by
\begin{equation}\label{GA0}
G=\langle \mathcal A\; | \; \mathcal O\rangle .
\end{equation}
Given a set of
words $\mathcal R$, we consider the quotient group of $G$
represented by
\begin{equation}\label{quot}
G_1= \langle \mathcal A\; |\;
\mathcal O\cup \mathcal R\rangle .
\end{equation}

\begin{figure}
  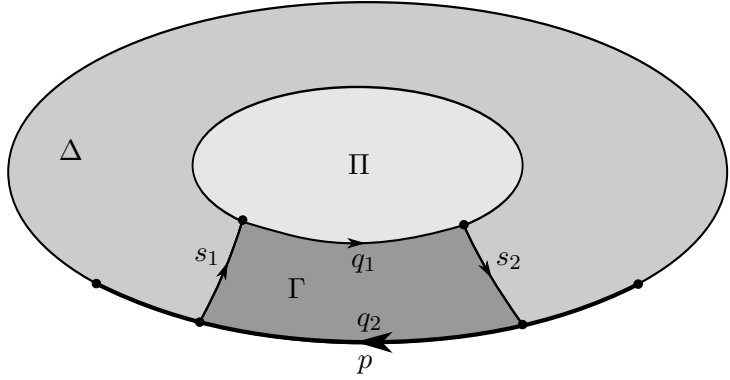
    \caption{Contiguity subdiagram. }\label{fig1}
\end{figure}

A cell in a van Kampen diagram over (\ref{quot}) is called an {\it
$\mathcal R$--cell} if its boundary label is a word from $\mathcal
R$. Let $\Delta $ be a van Kampen diagram over (\ref{quot}), $q$ a subpath of $\partial \Delta $, and $\Pi $ an $\mathcal R$--cell of $\Delta $. Suppose that there is a simple
closed path
\begin{equation}\label{pathp}
p=s_1q_1s_2q_2
\end{equation}
in $\Delta $, where $q_1$ is a subpath of
$\partial \Pi $, $q_2$ is a subpath of $q$, and
\begin{equation}\label{sides}
\max \{l( s_1),\, l(s_2) \} \le \e
\end{equation}
for some constant $\e >0$. By $\Gamma $ we denote the subdiagram of
$\Delta $ bounded by $p$. If $\Gamma $
contains no $\mathcal R$--cells, we say that $\Gamma $ is an {\it
$\e $--contiguity subdiagram} (or simply a contiguity subdiagram if $\e $ is fixed) of $\Pi $ to the subpath $q$ of $\partial \Delta $ and
$q_1$ is the {\it contiguity arc} of $\Pi $
to $q$. The ratio $l(q_1)/l(\partial \Pi )$ is called the {\it contiguity degree}
of $\Pi $ to $q$
and is denoted by $(\Pi , \Gamma , q)$. In case $q=\partial \Delta $, we talk about contiguity subdiagrams, etc., of $\Pi $ to $\partial \Delta $.  Since $\Gamma$ contains no $\mathcal R$--cells, it can be considered a diagram over (\ref{GA0}).

A van Kampen diagram $\Delta $ over (\ref{quot}) is said to be
{\it reduced} if $\Delta $ has minimal number of $\mathcal
R$--cells among all diagrams over (\ref{quot}) having the same
boundary label. When dealing with a diagram $\Delta $ over (\ref{quot}), it is convenient to consider the following transformations. Let $\Sigma $ be a subdiagram of $\Delta $ which contains no $\mathcal R$-cells, $\Sigma ^\prime $ another diagram over (\ref{GA0}) with $\lab (\partial \Sigma)\equiv\lab (\partial \Sigma ^\prime )$. Then we can remove $\Sigma $ and fill the obtained hole with $\Sigma ^\prime $. Note that this transformation does not affect $\lab(\partial \Delta )$ and the number of $\mathcal R$-cells in $\Delta $. If two diagrams over (\ref{quot}) can be obtained from each other by a sequence of such transformations, we call them {\it $\mathcal O$-equivalent}. \cite[Lemma 4.4]{Osi10} provides an analogue to the well-known Greendlinger Lemma for small cancellation over relatively hyperbolic groups. In this paper, we will make use of the more general version of this lemma appearing in the appendix of \cite{Osi10}.

\begin{lem}[\cite{Osi10}, Lemma 9.7]\label{rhGL}
Let $G$ be a group with presentation (\ref{GA0}). Suppose that the Cayley graph $\Gamma (G, \mathcal A)$ of $G$ is hyperbolic. Then for any $\lambda\in (0,1]$, $c\ge 0$, there exists $\e \ge 0$ such that for all  $\mu \in (0, 1/16]$, there exists $\rho >0$ with the following property.
Let $\mathcal R$ be a symmetrized set of words in $\mathcal A$
satisfying the $C(\e, \mu , \lambda , c , \rho )$--condition,
$\Delta $ a reduced van Kampen diagram over the presentation
(\ref{quot}) such that $\partial \Delta =q_1\cdots q_r$ for some $1\le r\le 4$, where $q_1, \ldots , q_r$ are $(\lambda , c)$-quasi-geodesic. Assume that $\Delta $ has at least one $\mathcal R$--cell. Then up to passing to an $\mathcal O$-equivalent diagram, then there is an $\mathcal R$-cell $\Pi $ of $\Delta $ and disjoint $\e $-contiguity subdiagrams $\Gamma _{j}$ of $\Pi $ to sections $q_j$, $j=1, \ldots , r$, of $\partial \Delta $ (some of them may be absent) such that $\sum\limits_{j=1}^r (\Pi, \Gamma _{j}, q_j)>1-13\mu $.
\end{lem}

This is actually a slight restatement of \cite[Lemma 9.7]{Osi10}, since we will need to choose $\e$ independent of $\mu$. However, this follows immediately from the choice of $\e$ in the proof of this lemma (see \cite[equation 36]{Osi10}).  In fact, aside from the inductive proof of this lemma, \cite{Osi10} only  makes use of the special case when $r=1$; we will need the more general statement for the proof of Lemma \ref{gl2}.


\section{Conjugacy and elementary subgroups in small cancellation quotients}


Throughout this section, let $G$ and $G_1$ be groups defined by (\ref{GA0}) and (\ref{quot}), respectively. We suppose $G$ is hyperbolic relative to $\Hl$, and $\mathcal A=X\cup\mathcal H$, where $X$ is a finite relative generating set of $G$ with respect to $\Hl$ and $\mathcal H$ is defined by (\ref{H}).  Let $\delta$ denote the hyperbolicity constant of $\Gamma (G, \mathcal A)$ provided by Theorem \ref{hms}.

The following lemma is a combination of Lemma 5.1 and Lemma 6.3 from \cite{Osi10}.

\begin{lem}\label{G1}
For any $\lambda\in (0,1]$, $c\ge 0$, $N>0$, there exist $\mu
>0$, $\e \ge 0$, and $\rho >0$ such that for any finite
symmetrized set of words $\mathcal R$ satisfying the $C_1(\e , \mu ,
\lambda , c, \rho )$--condition, the following hold.
\begin{enumerate}

\item[(a)] The group $G_1$ is hyperbolic
relative to the collection of images of subgroups $H_\lambda,
\lambda \in \Lambda $, under the natural homomorphism $G\to G_1$.

\item[(b)] The restriction of the natural homomorphism $G\to G_1$ to the subset of elements of length at most $N$ with
respect to the generating set $\mathcal A$ is injective.

\item[(c)]
Every element of finite order in $G_1$  is the image of an element of finite order of $G$.
\end{enumerate}
\end{lem}

For an element $g\in G$, the {\it translation number} of $g$ with respect to $\mathcal A$ is defined to be $$\tau _\mathcal A (g)=\lim\limits_{n\to \infty}\frac{|g^n|_\mathcal A }{n}.$$ This limit always exists and is equal to $\inf\limits_n (|g^n| \mathcal A /n) $ \cite{GS}.
The lemma below can be found in \cite[Theorem 4.43]{Osi06a}.

\begin{lem}\label{cyc}
There exists $d>0$ such that for any loxodromic element $g\in G$ we have $\tau _{X\cup \mathcal H} (g)\ge
d$.
\end{lem}

Given a word $W$ in $\mathcal A$, we say that a word $U$ is \emph{$W$-periodic} if it is a subword of $W^n$ for some $n\in \mathbb Z\setminus\{0\}$. In what follows, when speaking about Cayley graphs or van Kampen diagrams, we denote by {\it dist} the natural metric induced by identifying each edge with the segment $[0,1]$.

\begin{cor}\label{qg}
Suppose that $W$ is a word in $\mathcal A$ representing a loxodromic element $g\in G$ such that $|g|_{\mathcal A}=\|W\|\leq C$ for some $C>0$. Then any path in $\Gamma (G, \mathcal A)$  labeled by a $W$-periodic word is $(\frac{d}{C}, 2(C+d))$ quasi-geodesic.
\end{cor}
\begin{proof}
First observe that for any $n\in \mathbb N$,
\begin{equation}\label{qgc}
|g^n|_{\mathcal A}\geq n\inf_i\left(\frac{1}{i}|g^i|_{\mathcal A}\right)\geq nd\geq\frac{d}{C}n|g|_{\mathcal A}
\end{equation}
where $d$ is the constant from Lemma \ref{cyc}. Now suppose $p$ is a path labeled by a $W$-periodic word. Let $q$ be a maximal (maybe empty) subpath of $p$, labeled by $W^{n}$ for some $n\in \mathbb Z$ (we identify $W^0$ with the empty word). Then, $\l(p)\leq n |g|_{\mathcal A}+2C$. Combining this with \eqref{qgc} and the triangle inequality, we get
\[
\d(p_-, p_+)\geq \d(q_-, q_+)-2C= |g^n|_{\mathcal A}-2C\geq\frac{d}{C}n|g|_{\mathcal A}-2C\geq\frac{d}{C}\l(p)-2C-2d.
\]
Since a subword of a $W$-periodic word is also $W$-periodic, we are done.
\end{proof}

The next lemma provides a bound on contiguity degrees of $\mathcal R$-cells to paths with periodic labels and on a possible overlap between two contiguity subdiagrams to a geodesic if $\mathcal R$ satisfies a small cancellation condition.

\begin{lem}\label{gl1}
For any $\lambda \in (0, 1]$, $c\ge 0$, $\e>0$,  and $N\in\N$, there exist constants $D=D(\e,\lambda,c,\delta, N)$ and $\e_1\ge\e$ such that for all $\mu>0$ and $\rho>0$ and any set of words $\mathcal R$ satisfying the $C_1(\e_1 , \mu ,\lambda , c, \rho )$ condition, the following holds.
\begin{enumerate}
\item[(a)] Let $W$ be a word in $\mathcal A$ representing a loxodromic $g\in G$ such that $|g|_{\mathcal A}=\|W\|\leq N$. Then for any $R\in\mathcal R$ and any  quadrangle $Q=s_1q_1s_2q_2$ in $\Gamma(G, \mathcal A)$, where $\l(s_i)\leq\e$ for $i=1,2$, $\lab(q_1)$ is a subword of $R$, and $\lab(q_2)$ is $W$-periodic, we have $\l(q_1)\leq D\mu\|R\|+D$.
\item[(b)] Let $U$ and $V^{\pm 1}$ be disjoint subwords of some $R\in \mathcal R$, and let $r$ be a geodesic path in $\Gamma(G,\mathcal A)$. Suppose $q_1s_1r_1t_1$ and $q_2s_2r_2t_2$ are quadrangles in $\Gamma(G, \mathcal A)$ such that $\lab(q_1)\equiv U$, $\lab(q_2)\equiv V$, $r_1$, $r_2$ are subpaths of $r^{\pm 1}$, and $\l(s_i), \l(t_i)\leq\e$ for $i=1,2$. Then the overlap between $r_1$ and $r_2$ is at most $\mu\|R\|+\e_1$.
\end{enumerate}
\end{lem}
\begin{proof}
Without loss of generality we can assume that $s_1,s_2$ are geodesic. Since the $C_1$ condition becomes stronger as $\lambda$ increases and $c$ decreases, it suffices to assume that $\lambda\le \frac{d}{N}$ and $c\ge 2N+2d$. Thus $Q$ is a $(\lambda , c)$-quasi-geodesic quadrangle by Corollary \ref{qg}. Choose
$$
\e_1=2(K+\e),
$$
where $K=K(\lambda, c, \delta)$ is the constant provided by Lemma \ref{qgq}.

Our proof of part (a) will closely follow the ideas from the proof of \cite[Lemma 6.3]{Osi10}. Passing to a cyclic shift of $W^{\pm 1}$, we can assume $q_2$ is labeled by a prefix of $W^n$ for some $n\in \mathbb N$. We will derive a contradiction under the assumption that $q_1$ is sufficiently long; the exact constant $D$ can be easily extracted from the proof. First, note that the triangle inequality gives
\begin{equation}\label{ti}
\l(q_1)\leq \frac{1}{\lambda}(\d((q_1)_-, (q_1)_+)+c)\leq \frac{1}{\lambda}(\l(q_2)+2\e+c).
\end{equation}
Now, if $\l(q_2)\leq\frac{4}{3}\|W\|$ we have
\[
\l(q_1)\leq\frac{1}{\lambda}\left(\frac{4}{3}N+2\e\right)+c.
\]
Thus, it suffices to assume $\l(q_2)>\frac{4}{3}\|W\|$. Then we can decompose $\lab (q_2)$ as $\lab (q_2)\equiv UV_1UV_2$, where
\begin{equation}\label{lq2}
\frac{\lambda^2\l(q_2)}{5}\leq\|U\|\leq\frac{\lambda^2\l(q_2)}{4}
\end{equation}
and
\begin{equation}\label{lV1}
\|V_1\|>\frac{\l(q_2)}{3}.
\end{equation}

\begin{figure}
  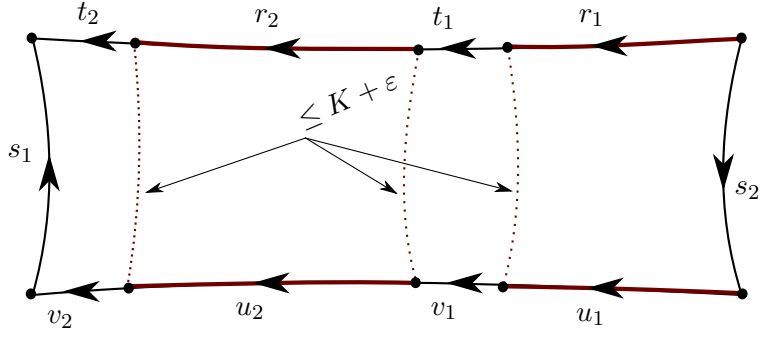
    \caption{Decompositions of $q_1$ and $q_2$ in the proof of Lemma \ref{gl1} (a).}\label{fig2}
\end{figure}

Let $q_2=u_1v_1u_2v_2$ be the corresponding decomposition of the path $q_2$ (see Fig. \ref{fig2}). Then by Lemma \ref{qgq}, we can find an initial subpath $r_1$ of $q_1^{-1}$ and a subpath $r_2$ of $q_1^{\pm 1}$ such that
\begin{equation}\label{sp}
\d((r_i)_{\pm}, (u_i)_{\pm})\leq K+\e
\end{equation}
for $i=1,2$. Now, we claim that for sufficiently long $q_1$, $r_1$ and $r_2$ will be disjoint. Indeed using (\ref{lq2}) we obtain
\[
\l(r_1)\leq\frac{1}{\lambda}(\d((r_1)_-, (r_1)_+)+c)\leq\frac{1}{\lambda}(\l(u_1)+2\e+2K+c)
\leq\frac{\lambda\l(q_2)}{4}+\frac{2\e+2K+c}{\lambda}.
\]
However, if $r_1$ contains $(r_2)_-$, then by (\ref{lV1}) we have
\[
\l(r_1)\geq \d((r_1)_-, (r_2)_-)\geq\lambda\l(u_1v_1)-c-2\e-2K\geq\frac{\lambda\l(q_2)}{3}-c-2\e-2K.
\]
These inequalities contradict each other for sufficiently large $\l(q_2)$, which can be ensured if $q_1$ is long enough by \eqref{ti}.

Thus, we can decompose $q_1^{-1}=r_1t_1r_2^{\xi}t_2$, for some $\xi=\pm 1$ and $t_1, t_2$ where at least $t_1$ is non-empty. Let $\lab(q_1)^{-1}\equiv R_1T_1R_2T_2$ be the corresponding decomposition of the label of $q_1^{-1}$. Then by \eqref{sp} we have $R_1=Y_1UZ_1$ and $R_2=Y_2U^{\pm 1}Z_2$ in $G$, where $\|Y_i\|,\|Z_i\|\leq K+\e$ for $i=1,2$. Thus, there exist $Y$, $Z$ such that $\|Y\|,\|Z\|\leq 2(K+\e)=\e_1$ and $R_1=YR_2^{\pm 1}Z$ in $G$. Now, since $\mathcal R$ satisfies the $C_1(\e_1 , \mu ,\lambda , c, \rho )$--condition and $R_1$, $R_2$ are disjoint subwords of $R$, we have that $\|R_1\|\leq\mu\|R\|$. Finally, using (\ref{lq2}) we obtain
\begin{align*}
\l(q_2)&\leq\frac{5}{\lambda^2}\|U\|\leq\frac{5}{\lambda^3}(\d((u_1)_-, (u_1)_+)+c)\\
&\leq\frac{5}{\lambda^3}(\l(r_1)+2\e+2K+c)\leq\frac{5}{\lambda^3}(\mu\|R\|+2\e+2K+c).
\end{align*}
Combining this with (\ref{ti}) produces a contradiction for sufficiently long $q_1$. This completes the proof of (a).

\begin{figure}
 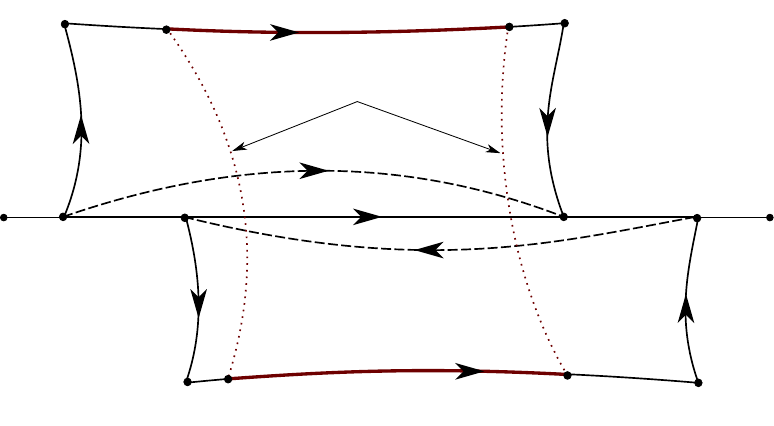
    \caption{The proof of Lemma \ref{gl1} (b). }\label{fig3}
\end{figure}

To prove (b) let $r'$ denote the overlap of $r_1$ and $r_2$ with arbitrary orientation (see Fig. \ref{fig3}). By Lemma \ref{qgq}, we can choose points $x_1$, $x_2$ on $q_1$ such that $\d((r')_-, x_1)\leq K+\e$ and $\d((r')_+, x_2)\leq K+\e$. Similarly we choose $y_1$ and $y_2$ on $q_2$ satisfying the same conditions. Now, if $x_1=x_2$ or $y_1=y_2$, then
\[
\l(r')=\d((r')_-, (r')_+)\leq 2(K+\e)=\e_1.
\]
Otherwise, we take $p_1$ to be the subpath of $q_1^{\pm 1}$ with endpoints $x_1, x_2$, and $p_2$ the subpath of $q_2^{\pm 1}$ with endpoints $y_1,y_2$. Then $\d((p_1)_{\pm}, (p_2)_{\pm})\leq 2(K+\e)$. Thus, by the $C_1(\e_1 , \mu ,\lambda , c, \rho )$--condition, we have that $l(p_1)\leq\mu\|R\|$. Thus,
\[
\l(r')=\d((r')_-, (r')_+)\leq\mu\|R\|+2(K+\e)=\mu \| R\| +\e_1.
\]
\end{proof}

The next lemma will allow us to describe maximal elementary subgroups corresponding to loxodromic elements of small length in small cancellation quotients of relatively hyperbolic groups (see the corollary after the lemma).

\begin{lem}\label{gl2}
For any $\lambda \in (0, 1]$, $c\ge 0$, and $N\in\N$,  there exist $\e_1>0$, $\mu>0 $, and $\rho >0$ such that the following holds. Suppose $\mathcal R$ satisfies the $C_1(\e_1, \mu ,\lambda , c, \rho )$--condition. Let $\alpha \colon G\to G_1$ be the natural epimorphism. Let $g$, $h$ be elements of $G$ such that $|g|_{\mathcal A}, |h|_{\mathcal A}\leq N$. Let $x\in G$ such that $\alpha(x^{-1}g^nxh^n)=1$ in $G_1$ but $x^{-1}g^nxh^n\ne 1$ in $G$ for some $n\in \mathbb N$; if $n\neq 1$, we further assume that $g$ and $h$ are loxodromic. Then there exists $y$ such that $| y|_{\mathcal A}<| x|_{\mathcal A}$, and $\alpha (y^{-1}g^nyh^n)=1 $ in $G_1$. Furthermore, $\alpha(x)\in\alpha(\langle g, y\rangle)$.
\end{lem}

\begin{proof}
 As in the proof of Lemma \ref{gl1}, it is sufficient to assume $\lambda\le\frac{d}{N}$, $c\ge 2N+2d$. Let $\e$ be chosen according to Lemma \ref{rhGL}.
Now choose $D$ and $\e_1$ according to Lemma \ref{gl1}; increasing $D$ if necessary, we also assume that $$D\geq\frac{2\e +N +c}{\lambda}.$$ Note that $D$ is independent of $\mu$ and $\rho$.  We will show that the conclusion of the lemma holds for sufficiently small $\mu$ and sufficiently large $\rho$.

Let $W$, $V$, and $X$ be shortest words in $\mathcal A$ representing $g$, $h$, and $x$ respectively. Let $\Delta$ be a reduced van Kampen diagram over \eqref{quot} with $\partial\Delta=p_1p_2p_3p_4$, $\lab(p_1)^{-1}\equiv\lab(p_3)\equiv X$, $\lab(p_2)\equiv W^n$, and $\lab(p_4)\equiv V^n$, where either $n=1$ or both $g$ and $h$ are loxodromic. Then $p_2$ and $p_4$ are $(\lambda, c)$ quasi-geodesics; this is obvious when $n=1$ and follows from Corollary \ref{qg} otherwise. Also $p_1$ and $p_3$ are geodesic paths by our choice of $X$. Since $x^{-1}g^nxh^n\neq 1$ in $G$, $\Delta$ must contain an $\mathcal R$-cell. Since $\e_1\ge\e$, $\mathcal R$ also satisfies $C_1(\e, \mu ,\lambda , c, \rho )$, hence we can apply  Lemma \ref{rhGL} for $\mu \in (0, 1/16]$ and large enough $\rho$. That is, passing to an $\mathcal O$-equivalent diagram if necessary, we can find an $\mathcal R$-cell $\Pi $ of $\Delta $ and disjoint $\e $-contiguity subdiagrams $\Gamma _{j}$ of $\Pi $ to $p_j$, $j=1, \ldots , 4$, (some of which may be empty) such that
\begin{equation}\label{sum1}
\sum\limits_{j=1}^4 (\Pi, \Gamma _{j}, p_j)>1-13\mu .
\end{equation}

We will now show that
\begin{equation}\label{sum2}
(\Pi, \Gamma _{2}, p_2)+(\Pi, \Gamma _{4}, p_4)\leq 2\left(D\mu+\frac{D}{\l(\partial\Pi)}\right).
\end{equation}
 If $g$ and $h$ are loxodromic, this follows from Lemma \ref{gl1}. When $n=1$, we let $\partial \Gamma _2= sptr$, where $p$ is a subpath of $p_2$, $r$ is a subpath of $\partial \Pi$, and $\max \{ \ell (s), \ell(t)\}<\e $. By the definition of the $C_1(\e, \mu ,\lambda , c, \rho )$ the path $r$ is $(\lambda, c)$-quasi-geodesic. Therefore,
$$
\ell(r)\le \frac1\lambda (\d(r_-, r_+)+c)\le\frac1\lambda(\ell (s) +\ell (p)+\ell(r)+c)\le \frac1\lambda(2\e +N +c)\le D.
$$
Consequently, $(\Pi, \Gamma _{2}, p_2)=\ell(r)/\l(\partial\Pi)\le D/\l(\partial\Pi)$. The analogous inequality holds true for $(\Pi, \Gamma _{4}, p_2)$ and we obtain
$$
(\Pi, \Gamma _{2}, p_2)+(\Pi, \Gamma _{4}, p_4)\leq  \frac{2D}{\l(\partial\Pi)}.
$$

Hence the inequality (\ref{sum2}) holds in both cases.

Combining (\ref{sum1}) and (\ref{sum2}) gives that for at least one of $i=1, 3$, we have
\begin{equation}\label{cd}
(\Pi, \Gamma_i, p_i)>\frac{1}{2}\left(1-(13+2D)\mu-\frac{2D}{\l(\partial\Pi)}\right).
\end{equation}

Without loss of generality, we assume this holds for $i=1$. Let $\partial\Gamma_1=s_1r_1t_1q_1$, where $\l(s_i)\leq\e$ for $i=1,2$, $q_1$ is a subpath of $\partial\Pi$, $r_1$ is a subpath of $p_1$. There are two cases to consider.

\smallskip

\begin{figure}
 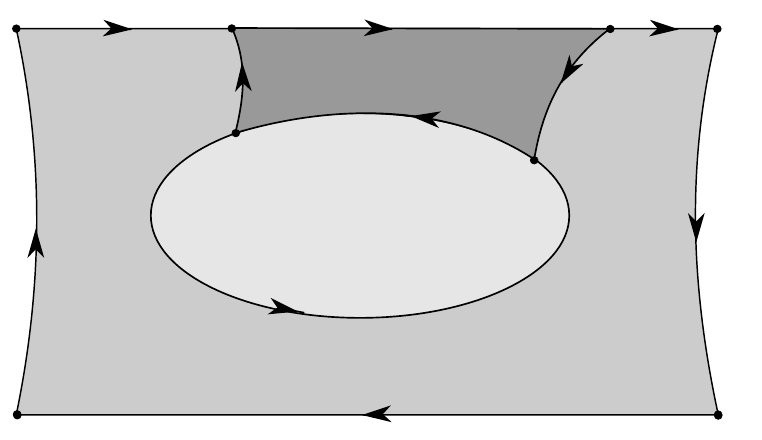
    \caption{Case 1 in the proof of Lemma \ref{gl2}. }\label{fig4}
\end{figure}

\smallskip

\begin{figure}
 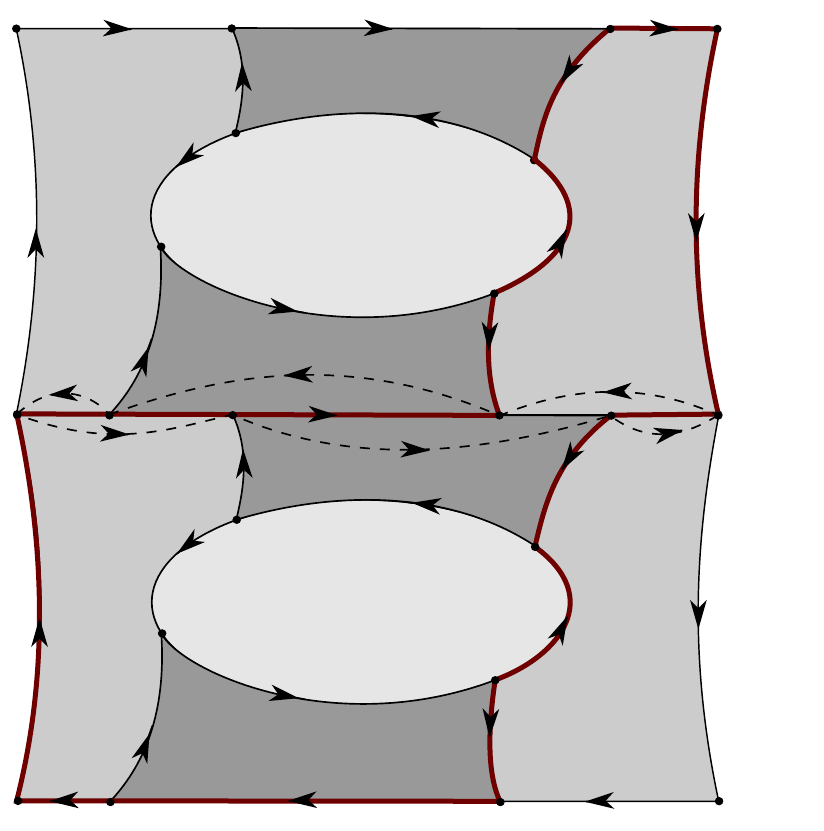
    \caption{Case 2 in the proof of Lemma \ref{gl2}. }\label{fig5}
\end{figure}

\noindent{\it Case 1.} First suppose $\Gamma_3$ is empty. Then $(\Pi, \Gamma_1, q_1)>\left(1-(13+2D)\mu-\frac{2D}{\l(\partial\Pi)}\right)$, so $\l(q_1)>(1-(13+2D)\mu)\l(\partial\Pi)-2D$. Let $\partial\Pi=q_1z$ (see Fig. \ref{fig4}). Then the path $s_1^{-1}zt_1^{-1}$ has the same start and end vertices as $r_1$. However,
\begin{align*}
\l(s_1^{-1}zt_1^{-1})&\leq \l(\partial\Pi)-\l(q_1)+2\e<\l(\partial\Pi)-(1-(13+2D)\mu)\l(\partial\Pi)+2D+2\e\\ & =(13+2D)\mu\l(\partial\Pi)+2D+2\e
\end{align*}
while
\begin{align*}
\l(r_1)& \geq \d((q_1)_-, (q_1)_+)-2\e\geq\lambda\l(q_1)-c-2\e\\ & >\lambda((1-(13+2D)\mu)\l(\partial\Pi)-2D)-c-2\e.
\end{align*}
Thus, for sufficiently small $\mu$ and suffiently large $\l(\partial\Pi)$ (which can be ensured by choosing large enough $\rho$), we will get that $\l(r_1)>\l(t_1zs_1)$. However, this contradicts the fact that $r_1$ is a subpath of a geodesic $p_1$. Thus, we can assume that $\Gamma_3$ is non-empty.

{\it Case 2.} Let $\partial\Gamma_3=s_2r_2t_2q_2$, where $\l(t_2), \l(s_2)\leq\e$, $q_2$ is a subpath of $\partial\Pi$, and $r_2$ is a subpath of $p_3$. Also, let $\partial\Pi=q_1z_1q_2z_2$. Now, we decompose $p_1=u_1r_1v_1$ and $p_3=u_2r_2v_2$. For definiteness we suppose that $\l(v_2)\le \l(u_1)$ (in the case $\l(v_2)\ge \l(u_1)$ the proof is similar). Let $\Delta ^\prime $ be a copy of $\Delta $. Given a path $a$ in $\Delta $, we denote its copy by $a^\prime$. Let us glue $\Delta $ and $\Delta ^\prime$ by identifying $p_3$ to $(p_1^\prime )^{-1}$(see Fig. \ref{fig5}). Let $j$ denote the (possibly empty) intersection of $r_1^\prime $ and $r_2^{-1}$. Consider the path $p=v_2^{-1}r_2^{-1}s_2^{-1}z_2t_1^{-1}v_1$. We want to show that
\begin{equation}\label{pisshorter}
\l(p)<\|X\|=\l(p_1).
\end{equation}

Observe that $\l(p)=\l(p_1)-\l(r_1)+\l(j) +\l(s_2^{-1}z_2t_1^{-1})$. Thus, we only need to show that $\l(r_1)>\l(j) +\l(s_2^{-1}z_2t_1^{-1})$. By Lemma \ref{gl1}, $\l(j)\leq\mu \l(\partial\Pi)+\e_1$. Also, $$\l(z_2)\leq\l(\partial\Pi)-\l(q_1)-\l(q_3)\leq(13+2D)\mu\l(\partial\Pi)+2D$$ by (\ref{sum1}) and (\ref{sum2}). Thus,
\begin{align*}
\l(j)+\l(s_2^{-1}z_2t_1^{-1}) & \leq\mu\l(\partial\Pi)+\e_1+(13+2D)\mu\l(\partial\Pi)+2D+2\e\\ & =(14+2D)\mu\l(\partial\Pi)+\e_1+2D+2\e.
\end{align*}
However, by \eqref{cd}
\begin{align*}
\l(r_1)& \geq \d((q_1)_-,(q_1)_+)-2\e\geq\lambda\l(q_1)-c-2\e>\\ & \frac{\lambda}{2}(1-(13+2D)\mu)\l(\partial\Pi)-\lambda D-c-2\e.
\end{align*}
Thus, we will have $\l(r_1)>\l(j)+\l(s_2^{-1}z_2t_1^{-1})$ as long as $\mu$ is sufficiently small and $\l(\partial\Pi)$ is sufficiently large; the later condition can be guaranteed by choosing sufficiently large $\rho$. This completes the proof of (\ref{pisshorter}).

Now let $y$ be the element of $G_1$ represented by $\lab(p)$. By (\ref{pisshorter}), we have $| y|_{\mathcal A}<| x|_{\mathcal A}$. Observe that $pp_2(p^\prime )^{-1}p^{\prime}_4$ is a closed path (it is represented by the bold line on Fig. \ref{fig5}), hence $yg^ny^{-1}h^n=\lab(p)\lab(p_2)\lab((p^\prime )^{-1})\lab(p^{\prime}_4)=1$ in $G_1$. Observe that $pp_2p_3$ is also a closed path, so $yg^nx=1$ in $G_1$, therefore $\alpha(x)\in\alpha(\langle g, y\rangle)$.
\end{proof}

The main result of this section is the following.

\begin{cor}\label{gc} For all $\lambda \in (0, 1]$, $c\ge 0$ and $N\in \N$, there exist  $\e_1>0$, $\mu>0 $, and $\rho >0$ such that if $\mathcal R$ satisfies the $C_1(\e_1, \mu ,\lambda , c, \rho )$ and $\alpha$ is the natural epimorphism from $G$ to $G_1$, then the following conditions are satisfied:
\begin{enumerate}
\item[(a)]
If $g, h\in B_{G, \mathcal A}(N)$, then $\alpha(g)\sim\alpha(h)$ if and only if $g\sim h$.

\item[(b)]
If $g\in B_{G, \mathcal A}(N)$ is loxodromic, then $E_{G_1}(\alpha(g))=\alpha(E_G(g))$.\end{enumerate}
\end{cor}
\begin{proof}
First, choose $\e_1$, $\mu$, $\rho$ satisfying the conditions of Lemma \ref{gl2}. Suppose $g$ and $h$ are non-conjugate elements of $G$ which become conjugate in $G_1$, and $g, h\in B_{G, \mathcal A}(N)$. Suppose $x$ is the shortest element in $G$ satisfying $\alpha(x^{-1}gxh^{-1})=1$. But then by Lemma \ref{gl2} there exists a strictly shorter element $y$ such that $\alpha(y^{-1}gyh^{-1})=1$, contradicting our choice of $x$. This proves (a).

Now suppose $g\in B_{G, \mathcal A}(N)$ is a loxodromic element such that $\alpha(E_G(g))\neq E_{G_1}(\alpha(g))$. Clearly, $\alpha(E_G(g))\subset E_{G_1}(\alpha(g))$. Let $x$ be a shortest element of $G$ such that $\alpha(x)\in E_{G_1}(\alpha(g))\setminus \alpha(E_G(g))$. Then by Lemma \ref{Eg} there exists some $n\in\mathbb N$ such that $\alpha (x^{-1}g^nxg^{\pm n})=1$ in $G_1$. Since $x\notin E_G(g)$, $x^{-1}g^nxg^{\pm n}\ne 1$ in $G$, so we can apply Lemma \ref{gl2} to find a strictly shorter element $y$ satisfying $\alpha (y^{-1}g^nyg^{\pm n})=1$ in $G_1$. Then $\alpha(y)\in E_{G_1}(g)$, and by our choice of $x$ we get that $\alpha(y)\in \alpha(E_G(g))$. However, then $\alpha(x)\in\alpha(\langle g, y\rangle)\le \alpha(E_G(g))$, a contradiction.
\end{proof}


\section{Constructing groups with specified conjugacy growth}


To prove our main result, we will need some special words satisfying sufficiently strong small cancellation conditions. These words were constructed in \cite{Osi10}. More precisely, let $G$ be a group hyperbolic
relative to a collection of subgroups $\Hl $, $X$ a finite
relative generating set of $G$ with respect to $\Hl $. Let $\delta$ be the hyperbolicity constant of
$\Gamma(G, X\cup\mathcal H)$ provided by Theorem \ref{hms}. Consider words $W$ satisfying the following conditions:

\begin{enumerate}
\item[(W$_1$)] $W\equiv xa_1b_1\ldots a_nb_n$ for some $n\ge 1$, where:

\item[(W$_2$)] $x\in X\cup \{ 1\} $;

\item[(W$_3$)] $a_1, \ldots , a_n$
(respectively $b_1, \ldots , b_n$) are elements of a parabolic subgroup $H_\alpha $
(respectively $H_\beta $), where $H_\alpha \cap H_\beta =\{ 1\} $;

\end{enumerate}

\begin{thm}[Theorem 7.5 \cite{Osi10}]\label{w}
There exists a constant $L=L(\e, \delta)>0$ and a finite set $\Omega\subset G$ such that the following is satisfied. Suppose that $W$ is a word in $X\cup\mathcal H$ satisfying the
conditions (W$_1$)--(W$_3$) and $a_i\ne a_j^{\pm 1}$,
$b_i\ne b_j^{\pm 1}$ whenever $i\ne j$, and $a_i\ne a_i^{-1}$,
$b_i\ne b_i^{-1}$, $i,j\in \{ 1, \ldots , n\}$. Also, suppose the elements $a_1, \ldots , a_n, b_1, \ldots , b_n$ do not  belong to the set $\{ g\in \langle \Omega \rangle\; : \; |g|_{\Omega} \le L \}$. Then the set $\mathcal W$ of all cyclic shifts of $W^{\pm 1} $ satisfies the $C_1(\e , \frac{3\e +11}{n} ,\frac13, 2, 2n+1)$ small cancellation
condition.
\end{thm}

\begin{thm}\label{scq}
Let $G$ be a group hyperbolic relative to a collection of
subgroups $\Hl $, $S$ a suitable subgroup of $G$, and $t_1, \ldots
, t_m$ arbitrary elements of $G$. Let $N\in\N$, and $X$ be a finite relative generating set of $G$. Then there exists a group $\overline{G}$ and an epimorphism
$\alpha \colon G\to \overline{G}$ such that:
\begin{enumerate}

\item[(a)] The group $\overline{G}$ is hyperbolic relative to $\{ \alpha
(H_\lambda ) \} _{\lambda \in \Lambda } $.

\item[(b)] For any $i=1,\ldots , m$, we have $\alpha (t_i)\in \alpha (S)$.

\item [(c)] $\alpha$ is injective on $B_{G, X\cup\mathcal H} (N)$. In particular, the restriction of $\alpha $ to $\bigcup_{\lambda\in \Lambda}H_\lambda $ is injective.

\item[(d)] $\alpha (S)$ is a suitable subgroup of $\overline{G}$.

\item[(e)] An element of $\overline{G} $ has finite order only if it is an image of an element of finite order in $G$. In particular, if $G$ is torsion free, then so is $\overline G$.

\item[(f)]
If $g, h\in B_{G, X\cup\mathcal H} (N)$, then $\alpha(g)\sim\alpha(h)$ if and only if $g\sim h$.

\item[(g)]
If $g\in B_{G, X\cup\mathcal H} (N)$ is loxodromic, then $E_{\overline{G}}(\alpha(g))=\alpha(E_G(g))$.

\end{enumerate}
\end{thm}

\begin{rem}\label{screm}
It is easy to see from the proof of Theorem \ref{scq} that $$\overline{G}= G/\ll t_1w_1, \ldots , t_mw_m\rr $$ for some elements $w_1, \ldots , w_m\in S$. Moreover, we can assume that $w_1, \ldots , w_m\in [S,S]$ by choosing the exponents $m_1, \ldots, m_n$ in the proof so that $m_1+\cdots + m_n=0$.
\end{rem}
Conditions $(a)-(e)$ of the theorem are proved in \cite{Osi10} using the theory of small cancellation over relatively hyperbolic groups. Additional lemmas proved in the previous section will allow us to prove the remaining conditions (f),(g) and to show that these conditions and (a)-(e) can be achieved simultaneously.

\begin{proof}[Proof of Theorem \ref{scq}]
We consider only the case $m=1$, since the general case will follow by induction from repeated applications of this case. We set $\mathcal A=X\cup\mathcal H$ and let $\mathcal O=\mathcal S\cup
\mathcal Q$ as defined in (\ref{Gfull}).

Let $\mu ,\e_1\ge \e, \rho $ be constants such that the conclusions of Lemma \ref{G1} and Corollary \ref{gc} hold for $\lambda =1/3 $, $c=2$,  and $N$. By Lemma \ref{non-com}, there are non-commensurable loxodromic elements $s_1, s_2\in S$ such that $E_G(s_1)=\langle s_1\rangle$ and $E_G(s_2)=\langle s_2\rangle$. Lemma \ref{Eg} gives that $G$ is hyperbolic relative to the collection $\Hl\cup E_G(s_1) \cup E_G(s_2)$. Then by Theorem
\ref{w}, there are $n$ and $m_1, \ldots , m_n$ such that
the set $\mathcal R$ of all cyclic shifts and their inverses of the  word
$$R\equiv ts_1^{m_1}s_2^{m_1}\ldots s_1^{m_n}s_2^{m_n}$$
satisfies the $C_1(\e_1, \mu, 1/3, 2, \rho )$--condition (and hence the $C_1(\e, \mu, 1/3, 2, \rho )$-condition as $\e_1\ge \e$). Indeed it suffices to choose large enough $n$ and $m_1, \ldots , m_n$ satisfying $m_i\ne \pm m_j$ whenever $i\ne j$. Let $\overline{G}$ be the quotient of $G$ obtained by
imposing the relation $R=1$ and $\alpha $ the corresponding natural
epimorphism.

Lemma \ref{G1} gives us assertions (a), (c), and (e). Note that in $\overline{G}$, the equality $R=1$ implies  $\alpha(t^{-1})=\alpha(s_1^{m_1}s_2^{m_1}\ldots s_1^{m_n}s_2^{m_n})$, hence $\alpha(t)\in \alpha(S)$. Assertions (f) and (g) follow from Corollary \ref{gc}. It remains to prove (d).  Without loss of generality we can assume that $s_1,s_2\in B_{G,X\cup\mathcal H}(N)$. Hence $E_{\overline G}(\alpha (s_1))\cap E_{\overline G}(\alpha (s_2))=\{1\}$ by (g). This means that the image of $S$ is a
suitable subgroup of $\overline{G}$.
\end{proof}

We will also need the following result which provides the building blocks for the groups constructed in the proof of Theorem \ref{main}.

\begin{thm}[Corollary 1.2 \cite{Osi10}]\label{2CC}
There exists a torsion free 2-generated group with exactly 2 conjugacy classes.
\end{thm}
\begin{rem}\label{expgr}
Note that every torsion free group $G$ with $2$ conjugacy classes has exponential growth. Indeed every element $g\in G$ is conjugate to its square. If $g\ne 1$, this easily implies that the intersection of the cyclic subgroup $\langle g\rangle $ with a ball of radius $n$ with respect to a fixed finite generating set of $G$ has exponentially many elements.
\end{rem}

Groups with $2$ conjugacy classes are constructed in \cite{Osi10} as direct limits of relatively hyperbolic groups. The proof of Theorem \ref{main} is based on the same idea; however its implementation is not automatic. Before proceeding to the proof of Theorem \ref{main}, we give a sketch of the construction used in \cite{Osi10} and indicate the main difficulties which occur in the proof of our main result.

Let $R$ be a countable torsion free group in which all non-trivial elements are conjugate. This group can be easily  constructed using successive HNN-extensions (see \cite{HNN} or Theorem 3.3 in \cite[Chapter 4]{LS}). Let $F=F(x,y)$ be the free group on two generators, and consider the free product $G(0)=R*F$, which is hyperbolic relative to $R$. Enumerate all elements of $G(0)$ as $\{1=g_0, g_1,...\}$ and all elements of $R$ as $\{1=r_0,r_1,...\}$. Now we inductively create a sequence of groups and epimorphisms $G(0)\to G(1)\to ...$ as follows. After we have constructed a group $G(i)$, which is assumed to be hyperbolic relative to the image of $R$, we take an HNN-extension with a stable letter $t$ which conjugates $g_{i+1}$ to some non-trivial element of $R$ (unless $g_{i+1}$ is already parabolic and then we skip this step). Then we apply parts (a)-(e) of Theorem \ref{scq} to this group with the image of $F$ as a suitable subgroup and $\{r_i, t\}$ as the finite set of elements. The resulting group is $G(i+1)$, and since the image of $t$ is inside the image of $F$, there is a natural quotient map from $G(i)$ to $G(i+1)$. Thus, the direct limit of this sequence will be a quotient of $G(0)$, and since $G(0)$ is generated by $\{x, y, r_1,...\}$ and the image of each $r_i$ is inside the image of $F$ in $G(i+1)$, the limit group will be generated by $\{x, y\}$. Since each $g_i$ is conjugate to an element of $R$ in $G(i+1)$, all non-trivial elements will be conjugate in the limit group.

To prove Theorem \ref{main}, instead of trying to make all elements conjugate we want to control the number of conjugacy classes inside each ball with respect to a fixed finite generating set. So at the $i$th step of our construction we fix the desired number of conjugacy classes on the sphere of radius $i$ (up to some constants), making all other elements of the sphere conjugate. The main problem, however, is that the conjugacy relations which we want to add may also produce ``unwanted" conjugacy relations between elements we want to keep unconjugate. For instance conjugating two elements $x$ and $y$, we also make $x^n$ conjugate to $y^n$ for all $n$. Induced conjugations of this particular type can be controlled by working with primitive conjugacy classes and making all elements in our group conjugate to all their nontrivial powers. However this does not solve the problem completely as ``unwanted" conjugations can occur even between primitive elements. More precisely, the problem splits into two parts. When dealing with the sphere of radius $i$ at step $i$, we have to make sure that
\begin{enumerate}
\item[1)]  ``Unwanted" conjugations do not occur inside the ball of radius $(i-1)$.
\item[2)]   We keep enough non-conjugate primitive elements on spheres of radii $> i$ to continue the construction.
\end{enumerate}

To overcome the first difficulty we ``attach" a new parabolic subgroup with $2$ conjugacy classes to a representative of each conjugacy class which we want to keep inside the ball of radius $(i-1)$. Then Lemma \ref{malnorm} together with part (a) of Theorem \ref{scq} guarantee that such classes remain different at all steps of the inductive construction, and hence in the limit group.

The second part of the problem is more complicated and is typical for such inductive proofs. It is, in fact, the main obstacle in implementing the ideas from \cite{Osi10} in the proof of Theorem \ref{main} and is the reason why we need to go deep in small cancellation theory and add new parts (f) and (g) to Theorem \ref{scq}. To guarantee~2) we construct sets $U_i$ of elements with ordinary word length $i$ but relative length at most $4$.  Then parts (f) and (g) of Theorem \ref{scq} come into play and allow us to control these elements during the small cancellation substep of each step; Lemma \ref{conjHNN} is used to control them during the HNN-extension substep.

\begin{thm}\label{main}
Let $G$ be a group generated by a finite set $X$, $f$ the conjugacy growth function of $G$ with respect to $X$. Then the following conditions hold.
\begin{enumerate}
\item[(a)] $f$ is non-decreasing.
\item[(b)] There exists $a\ge 1$ such that $f(n) \le  a^n$ for every $n\in\N$.
\end{enumerate}
Conversely, suppose that a function $f\colon\N\to \N$ satisfies the above conditions (a) and (b). Then there exists an infinite finitely generated group $G$ such that $\xi_G\sim f$.
\end{thm}

\begin{proof} The `only if' part of the theorem is obvious. Let us prove the other one.
Suppose $f\colon\N\to\N$ is a non-decreasing function such that $f\le a^n$ for some $a>1$. If $f\equiv 1$, the statement is obvious. Otherwise passing to an equivalent function if necessary, we can assume that $f(n)\ge 2$ for all $n$. Let $\bar{f}$ denote the function defined by $\bar{f}(n)=f(n)-f(n-1)$.

Let $A$ be a finitely generated torsion free group with two conjugacy classes. Clearly, it suffices to assume that $f(n)\leq\gamma_A(n)$, since $\gamma_A$ is exponential by Remark \ref{expgr}.  Set $G(1)=A*\langle h\rangle$, where $h$ generates an infinite cyclic group. Let $X'$ be a finite generating set for $A$. Then we take $X=X'\cup h^{-1}X'h\cup\{h\}$ as a finite generating set for $G(1)$. Let $B=h^{-1}Ah$, and fix $a_0\in A$ and $b_0\in B$ such that $|a_0|_X=|b_0|_X=1$.

Further for each $i\geq 2$, we create a collection of subsets $U_i=\{ab : a\in A\setminus\{1\}, b\in B\setminus\{1\}, |ab|_X=i\}$. Note that all elements of $U=\bigcup_{i=1}^{\infty} U_i\cup\{a_0\}$ have word length at most $4$ with respect to the generators $A\cup\{h\}$. Clearly $|U_n|\succeq \gamma_A(n)$, and since $\gamma_A$ is exponential, there exists a constant $L$ such that
\begin{equation}\label{ULn}
|U_{Ln}|\geq(\gamma_{G(1)}(n-1))(\gamma_{G(1)}(n-1)+1)+\bar{f}(n).
\end{equation}

Suppose we have constructed a group $G(k)$, an epimorphism $\phi_k\colon G(1)\to G(k)$, and a collection of subsets $\{a_0\}=W_1\subset...\subset W_k\subset U$, such that the following conditions are satisfied.
\begin{enumerate}[(a)]
\item \label{i}
$\phi_k$ is injective on $X$, $A$, and $U$ (so we identify these sets with their images in $G(k)$).
\item \label{hr}
$G(k)$ is hyperbolic relative to a collection $\mathcal C_k$ of proper subgroups with two conjugacy classes.
\item $G(k)$ is a suitable subgroup of itself.
\item \label{tf}
$G(k)$ is torsion free.
\item \label{g}
Every $g\in B_{G(k),  X}(k-1)$ is parabolic in $G(k)$.
\item \label{pc}
Each element of $W_k$ is parabolic, and there is exactly one element of $W_k$ inside each parabolic conjugacy class. In particular, distinct elements of $W_k$ are non-conjugate.
\item \label{f}
For all $1\leq n\leq k$, $|W_n|=f(n)-1$, and for all $w\in W_n$, $|w|_X\leq Ln$.
\item \label{nc2}
if $u, v$ are two different elements of $U^{\pm 1}$ and $u\sim v$ in $G(k)$, then $u\sim v\sim a_0$ in $G(k)$. Furthermore at most $\gamma_{G(1)}(k-1)$ elements of $U$ are conjugate to $a_0$ in $G(k)$.
\item\label{pe}
For all $u\in U$ such that $u$ is loxodromic in $G(k)$, $u$ is also primitive in $G(k)$. In particular, $E_{G(k)}(u)=\langle u\rangle$.
\end{enumerate}

Obviously (a)-(f) hold for $G(1)$ with $\phi_1$ the identity map and $\mathcal C_1=\{A\}$ (see Lemma \ref{non-com} for (c)). Passing to an equivalent function, we can assume that $f(1)=2$ without loss of generality. This gives (g) for $G(1)$. It is clear (e.g. from \cite[Chapt. IV, Theorem 1.4]{LS}) that all elements of $U$ are pairwise non-conjugate in $G(1)$. If $u\in U$ is loxodromic in $G(1)$, then $u\neq a_0$ and so $u=a_1h^{-1}a_2h$ for some $a_1, a_2\in A\setminus\{1\}$. The normal form theorem for free products \cite[Chapt. IV, Theorem 1.2]{LS} implies that $u$ is primitive in $G(1)$. Thus (h) and (i) also hold for $G(1)$.

Now we construct $G(k+1)$ in a sequence of four steps. The intermediate groups constructed in each step will be denoted as follows.
\[
G(k) \stackrel{\iota_1}\hookrightarrow G^{\prime}(k) \stackrel{\alpha_1}\twoheadrightarrow G^{\prime\prime}(k) \stackrel{\iota_2}\hookrightarrow G^{\prime\prime\prime}(k) \stackrel{\alpha_2}\twoheadrightarrow G(k+1).
\]
Here $\iota_1$ will be the natural embedding into an HNN-extension of the previous group, while $\iota_2$ will be the natural embedding into an HNN-extension of a a free product, where the previous group is one of the factors. $\alpha_1$ and $\alpha_2$ will be epimorphisms which will correspond to taking a small cancellation quotient of the previous group. We will first show how to construct the group $G(k+1)$, and then verify that it satisfies all the inductive conditions.

{\bf Step 1.} Let $g_1,...,g_n$ be the list of all elements in $G(k)$ such that $|g_i|_X=k$ and $g_i$ is loxodromic in $G(k)$ for each $1\leq i\leq n$. Note that $n\leq \bar{\gamma}_{G(1)}(k)$. Since $G(k)$ is torsion free, for each $i$ there exists some $h_i$ such that $E_{G(k)}(g_i)=\langle h_i\rangle$. Now we define $G^{\prime}(k)$ as the multiple HNN-extension
\[
G^{\prime}(k)=\langle G(k), t_1,...,t_n \; | \; h_i^{t_i}=a_0\rangle.
\]

Let $\iota_1\colon G(k)\hookrightarrow G^{\prime}(k)$ be the natural embedding; for convenience we identify $G(k)$ with its image in $G^{\prime}(k)$.  Suppose $u, v\in U^{\pm 1}$ such that $u\sim v$ in $G^{\prime}(k)$. If $u\sim v$ in $G(k)$, then by \eqref{nc2} $u\sim v\sim a_0$ in $G(k)$ and hence also in $G^{\prime}(k)$. Otherwise, by Lemma \ref{conjHNN} either $u$ or $v$ is conjugate to some element $h_i^m$ in $G(k)$. If this holds for $u$, we have $u\sim h_i^m\sim a_0^m\sim a_0$ in $G^{\prime}(k)$, and similiarly if it holds for $v$. Thus, two elements of $U^{\pm 1}$ are either non-conjugate in $G^{\prime}(k)$ or they are both conjugate to $a_0$.

Further if $u\sim a_0$ in $G^{\prime}(k)$ but not in $G(k)$, then $u$ must be conjugate to a power of some $h_i$ in $G(k)$ by Lemma \ref{conjHNN}. For each $h_i$, $u\sim h_i^m$ in $G(k)$ implies that $m=\pm 1$ since $u$ is primitive in $G(k)$ by \eqref{pe}. By \eqref{nc2}, for each $1\leq i\leq n$, there is at most one element $u\in U$ conjugate to $h_i^{\pm 1}$ in $G(k)$.  Thus, the number of elements of $U$ conjugate to $a_0$  in $G^{\prime}(k)$ is at most $\gamma_{G(1)}(k-1)+n\leq\gamma_{G(1)}(k)$. By Lemma \ref{conjHNN} and \eqref{pe}, all elements $u\in U$ which are loxodromic in $G^{\prime}(k)$ are primitive in $G^{\prime}(k)$. Corollary \ref{sgHNN} gives that $G^{\prime}(k)$ will be hyperbolic relative to $\mathcal C_k$ and $G(k)$ is suitable in $G^{\prime}(k)$.

{\bf Step 2.}
Let $G^{\prime\prime}(k)$ be the quotient group of $G^{\prime}(k)$ provided by applying Theorem \ref{scq} to $G^{\prime}(k)$ with $\{t_1,...,t_n\}$ as our finite set, $G(k)$ as our suitable subgroup, and $N=4$. Let $\alpha_1\colon G^{\prime}(k)\twoheadrightarrow G^{\prime\prime}(k)$ be the corresponding epimorphism.

Since elements of $U\cup X$ all have relative length at most $4$, $\alpha_1$ will be injective on $U\cup X$, so we identify these sets with their images. The final two assertions of Theorem \ref{scq} give that two elements of $U^{\pm 1}$ are conjugate in $G^{\prime\prime}(k)$ if and only if they were conjugate in $G^{\prime}(k)$, and for each loxodromic $u\in U$, $\langle u\rangle=E_{G^{\prime}(k)}(u)=E_{G^{\prime\prime}(k)}(u)$. Hence, all loxodromic elements of $U$ are still primitive in $G^{\prime\prime}(k)$. Let us prove that $\alpha_1\circ\iota_1$ is surjective. Since $G^{\prime}(k)$ is generated by $G(k)\cup\{t_1,...,t_n\}$ and for each $1\leq i\leq n$, $\alpha_1(t_i)\subset \alpha_1(G(k))$, we have that $G^{\prime\prime}(k)$ is generated by $\alpha_1(G(k))$. Thus, $\alpha_1\circ\iota_1$ will be surjective, and $G^{\prime\prime}(k)$ will be finitely generated by (the image of) $X$. Theorem \ref{scq} also gives that $\alpha_1(G(k))=G^{\prime\prime}(k)$ will be a suitable subgroup of $G^{\prime\prime}(k)$.

Now let $U^{\prime}_{L(k+1)}$ be the set of $u\in U$ such that $|u|_X=L(k+1)$ in $G^{\prime\prime}(k)$. Then $|U^{\prime}_{L(k+1)}|\geq|U_{L(k+1)}|$; this follows from the fact that for each $u\in U_{L(k+1)}$, we have $|u|_X\leq L(k+1)$ in $G^{\prime\prime}(k)$, so we can choose $j\in\mathbb N$ such that $|ub_0^j|_X=L(k+1)$. An element $u\in U^{\prime}_{L(k+1)}$ will be called {\it good} if for all elements $v$ conjugate to $u$ in $G^{\prime\prime}(k)$, we have $|v|_X\geq k+1$; otherwise it will be called {\it bad}. We want to show that $U^{\prime}_{L(k+1)}$ contains at least $\bar{f}(k+1)$ good elements.

Indeed otherwise by (\ref{ULn}), $U^{\prime}_{L(k+1)}$ must contain $(\gamma_{G(1)}(k))(\gamma_{G(1)}(k)+1)$ bad elements, each of which is conjugate to some element of $X$-length at most $k$. Since there are at most $\gamma_{G(1)}(k)$ such elements in $G^{\prime\prime}(k)$, there exists $V\subset U^{\prime}_{L(k+1)}$ such that $V$ contains $(\gamma_{G(1)}(k)+1)$ pairwise conjugate elements. Then all elements of $V$ must be pairwise conjugate in $G^{\prime}(k)$, and thus all elements of $V$ are conjugate to $a_0$ in $G^{\prime}(k)$. But this contradicts the fact that there are at most $\gamma_{G(1)}(k)$ elements of $U$ conjugate to $a_0$ in  $G^{\prime}(k)$. Thus $U^{\prime}_{L(k+1)}$ contains at least $\bar{f}(k+1)$ good elements.

{\bf Step 3.}
Let $W^{\prime}_{k+1}=\{w_1,...,w_s\}$ be a subset of the good elements of $U^{\prime}_{L(k+1)}$ such that $s=|W^{\prime}_{k+1}|=\bar{f}(k+1)$.  Note that if $u$ is a good element, then $u$ is not conjugate to $a_0$, hence $u$ is not conjugate to any other element of $U^{\pm 1}$. Thus all elements of $W^{\prime}_{k+1}$ are loxodromic and hence primitive; furthermore, they are pairwise non-commensurable by Lemma \ref{primcom}. Then we define $W_{k+1}=W_k\cup W^{\prime}_{k+1}$. Now, for each $1\leq i\leq s$, let $C_i$ be a torsion free group with two conjugacy classes, generated by $\{x_i, y_i\}$. Consider the group  $G^{\prime\prime}(k)*(*_{i=1}^{s}C_i)$, which naturally contains an isometrically embedded copy of  $G^{\prime\prime}(k)$. By Lemma \ref{relsub}, this group will be hyperbolic relative to $\mathcal C_{k+1}$, where $\mathcal C_{k+1}=\mathcal C_k\cup(\cup_{i=1}^{s}C_i)$. Also, clearly primitive elements of $G^{\prime\prime}(k)$ remain primitive in $G^{\prime\prime}(k)*(*_{i=1}^{s}C_i)$, and any two non-conjugate elements of $G^{\prime\prime}(k)$ remain non-conjugate. Since each $w_i$ is primitive and loxodromic, we get that $E_{G^{\prime\prime}(k)*(*_{i=1}^{s}C_i)}(w_i)=\langle w_i\rangle$.  Now we take a multiple HNN-extension and form the group
\[
G^{\prime\prime\prime}(k)=\langle G^{\prime\prime}(k)*(*_{i=1}^{s}C_i),d_1,...,d_s | w_i^{d_i}=x_i\rangle.
\]
Let $\iota_2\colon G^{\prime\prime}(k)\hookrightarrow G^{\prime\prime\prime}(k)$ denote the natural embedding, and again we identify $G^{\prime\prime}(k)$ with its image. Since the elements $w_i$ are pairwise non-commensurable, by Lemma \ref{conjHNN} we can inductively apply Corollary \ref{sgHNN} to get that $G^{\prime\prime\prime}(k)$ is hyperbolic relative to $\mathcal C_{k+1}$ and contains $G^{\prime\prime}(k)$ as a suitable subgroup.

{\bf Step 4.}
Finally, we obtain $G(k+1)$ as the quotient group of $G^{\prime\prime\prime}(k)$ by applying Theorem \ref{scq} to the finite set $\{ d_i, x_i, y_i:1\leq i\leq s\}$, suitable subgroup $G^{\prime\prime}(k)$, and $N=4$. Let $\alpha_2\colon G^{\prime\prime\prime}(k)\twoheadrightarrow G(k+1)$ be the corresponding epimorphism, and define $\phi_{k+1}=\alpha_2\circ\iota_2\circ\alpha_1\circ\iota_1\circ \phi _k$. Let us prove that $\phi_{k+1}$ is surjective. We have shown that $\alpha_1\circ\iota_1$ is surjective.  Similarly, since $G^{\prime\prime\prime}(k)$ is generated by $G^{\prime\prime}(k)\cup\{d_i, x_i, y_i: 1\leq i\leq s\}$ and $\alpha_2(\{d_i, x_i, y_i: 1\leq i\leq s\})\subset\alpha_2(G^{\prime\prime}(k))$, we get that $G(k+1)$ is generated by $\alpha_2(G^{\prime\prime}(k))$, hence $\alpha_2\circ\iota_2$ is surjective. Thus, $\phi_{k+1}$ is surjective being the composition of surjective maps.

 Let us now prove that $G(k)$ satisfies all the inductive assumptions. First, Theorem \ref{scq} gives that $\alpha_1$ and $\alpha_2$ are injective on all elements of relative length at most $4$, which includes all elements in $U$, $X$ and $\mathcal C_{k+1}$. Hence $\phi_ {k+1}$ will be injective on these sets being the composition of these maps and the injective maps $\iota_1$ and $\iota_2$. $G^{\prime\prime\prime}(k)$ is hyperbolic relative to $\mathcal C_{k+1}$, and Theorem \ref{scq} gives that $G(k+1)$ will be hyperbolic relative to $\mathcal C_{k+1}$ and $\alpha_2(G^{\prime\prime}(k))=G(k+1)$ is a suitable subgroup of itself.  Taking HNN-extensions and free products of torsion free groups gives torsion free groups, and combining this with Theorem \ref{scq} gives that $G(k+1)$ will be torsion free. Clearly, every element of $B_{G(k), X}(k-1)$ is parabolic in $G^{\prime}(k)$, thus they are also parabolic in $G(k+1)$.  By construction, each $w\in W_{k+1}$ is parabolic in $G^{\prime\prime\prime}(k)$ and the conjugacy class of $w$ corresponds to a unique parabolic subgroup. The definition of $W_{k+1}$ gives that $|W_{k+1}|=f(k+1)-1$, and for all $w\in W_{k+1}$, $|w|_X\leq L(k+1)$; clearly this also holds for all $1\leq n\leq k$ as passing to quotient groups can only decrease word length. We have shown that in $G^{\prime}(k)$, two elements of $U^{\pm 1}$ are conjugate if and only if they are both conjugate to $a_0$, and furthermore at most $\gamma_{G(1)}(k-1)$ elements of $U$ are conjugate to $a_0$ in $G^{\prime}(k)$.  At all other steps non-conjugate elements of $U^{\pm 1}$ remain non-conjugate, so this also holds in $G(k+1)$. Finally, loxodromic elements of $U$ are primitive in $G^{\prime\prime}(k)*(*_{i=1}^{s}C_i)$, and Lemma \ref{conjHNN} gives that they are primitive in $G^{\prime\prime\prime}(k)$. Hence for all loxodromic $u\in U$, $\langle u\rangle=E_{G^{\prime\prime\prime}(k)}(u)=E_{G(k+1)}(u)$ by Theorem \ref{scq}, so $u$ is still primitive in $G(k+1)$. Thus, $G(k+1)$ satisfies all the inductive conditions.

Now, we take $G$ to be the limit of this sequence of groups; that is, let $G=G(1)/N$,
where $N=\bigcup_{i=1}^\infty \Ker \phi_i$. We will show that every conjugacy class in $G$ has a representative in $\bigcup_{k=1}^{\infty} W_k$. Suppose $g\in B_{G, X}(n)$, and let $g_0$ be a pre-image of $g$ in $G(1)$ such that $|g_0|_X\leq n$. Then $g_0$ is parabolic in the group $G(n+1)$ by condition \eqref{g}, so $g_0$ is conjugate to an element of $W_{n+1}$ in $G(n+1)$ by \eqref{pc}. Hence $g$ is conjugate to an element of $W_{n+1}$ in $G$. Thus we have $$\xi_G(n)\leq |W_{n+1}|+1=f(n+1)\leq f(2n).$$ On the other hand, all elements of $W_n$ are pairwise non-conjugate, and for each $w\in W_n$, $|w|_X\leq Ln$. Hence $f(n)=|W_n|+1\leq\xi_G(Ln)$.  Therefore $\xi_G\sim f$.
\end{proof}


\section{Conjugacy growth and subgroups of finite index}


We now move to the proof of Theorem \ref{main2}. We start with an `infinitely generated version' of Theorem \ref{main2}.

\begin{lem}\label{NCZ}
There exists a short exact sequence $$1\rightarrow N\rightarrow C \rightarrow \mathbb Z_2\rightarrow 1$$ such that the following hold.
\begin{enumerate}
\item[(a)] The group $C$ is countable and torsion free.
\item[(b)] The subgroup $N$ has exactly $2$ conjugacy classes.
\item[(c)] There is a free subgroup $F\le N$ of rank $2$ and an element $a\in C$ such that for any two distinct elements $f_1,f_2\in F$, $af_1$ and $af_2$ are not conjugate in $C$.
\end{enumerate}
\end{lem}

\begin{proof}
We proceed by induction. Let $A_0=\langle a,b,c\rangle$ be the free group of rank $3$ and let $\e_0\colon A_0\to \langle a\mid a^2=1\rangle \cong \mathbb Z_2$ be the natural epimorphism. Assume that $A_n$ is already constructed together with an epimorphism $$\e_n\colon A_n\to  \mathbb Z_2.$$  Let $K_n$ denote the kernel of $\e_n$. We enumerate all elements of $K_n=\{ 1, k_0,k_1, \ldots \} $ and let $A_{n+1}$ be the multiple HNN-extension
$$
\langle A_n, \{ t_i\} _{i\in \mathbb N} \mid k_i^{t_i}=k_0\rangle.
$$
The map sending $K_n$ and all stable letters to $1$ (here $1$ denotes the identity element of $\Z_2$) extends to a homomorphism $\e_{n+1}\colon A_{n+1}\to \mathbb Z_2$.

Let $C=\bigcup\limits_{n=0}^\infty A_n$ and $N=\bigcup\limits_{n=0}^\infty K_n$. Clearly $N$ is a normal subgroup of index $2$ in $C$. Since all nontrivial elements of $K_n$ are conjugate in $K_{n+1}$, $N$ has exactly $2$ conjugacy classes. On the other hand, Lemma \ref{conjHNN} implies by induction that for any distinct $f_1,f_2\in \langle b,c\rangle$, the elements $af_1$ and $af_2$ are not conjugate in $A_n$. Hence the same holds true in $C$.
\end{proof}

\begin{thm}\label{main2}
There exists a finitely generated group $G$ and a finite index subgroup $H\le G$ such that $H$ has $2$ conjugacy classes while $G$ is of exponential conjugacy growth.
\end{thm}

\begin{proof}
Let $$1\rightarrow N\rightarrow C \stackrel{\e}\rightarrow \mathbb Z_2\rightarrow 1$$ be the short exact sequence provided by Lemma \ref{NCZ}. The desired group $G$ is constructed as an inductive limit of
relatively hyperbolic groups as follows. Let $$G(0)=C\ast
F(x,y),$$ where $F(x,y)$ is the free group of rank $2$ generated
by $x$ and $y$. We enumerate all elements of $$C=\{ 1=c_0, c_1,
c_2, \ldots \} $$ and $$G(0)=\{ 1=g_0, g_1, g_2, \ldots \} .$$ Without loss of generality we may assume that
\begin{equation}\label{c1c2}
\e (c_1)=1.
\end{equation}
Here we use multiplicative notation and $1$ denotes the trivial element of $\mathbb Z_2$.

Suppose that for some $i\ge 0$, the group $G(i)$ has already been
constructed together with an epimorphism $\phi_i\colon G(0)\to
G(i)$ and an epimorphism $\alpha _i\colon G(i)\to \mathbb Z_2$. We use the same notation for elements $x,y, c_0, c_1, \ldots , g_0, g_1, \ldots $ and their images in
$G(i)$. Assume that $G(i)$ satisfies the following conditions. It
is straightforward to check these conditions for $G(0)$, the
identity map $\phi _0\colon G(0)\to G(0)$, and the epimorphism $\alpha _0\colon G(0)\to \mathbb Z_2$ which is induced by $\e$ and the map sending $x$ and $y$ to the non-trivial element of $\mathbb Z_2$.
\begin{enumerate}
\item[(a)] The restriction of $\phi _i$ to the subgroup $C$ is
injective. In what follows we identify $C$ with its image in
$G(i)$.

\item[(b)] $G(i)$ is hyperbolic relative to $C$.

\item[(c)] The elements $x$ and $y$ generate a suitable subgroup
of $G(i)$.

\item[(d)] $G(i)$ is torsion free.

\item[(e)] In $G(i)$, the elements $c_0, \ldots , c_i$
are contained in the subgroup generated by $x$ and $y$.

\item[(f)] The diagram
$$
\begin{CD}
G(0) @>{\alpha _0}>> \mathbb Z_2\\
@V{\phi_i}VV @VV{\rm id}V\\
G(i) @>{\alpha _i}>> \mathbb Z_2
\end{CD}
$$
is commutative.

\item[(g)] In $G(i)$, for every $j=1, \ldots , i$, if $\alpha_i(g_j)=1$ then the element $g_j$ is conjugate to $c_1$ by an element of $\Ker \alpha _i$ .

\end{enumerate}
The group $G(i+1)$ is obtained from $G(i)$ in two steps.

{\bf Step 1.} If $g_{i+1}$ is a parabolic element
of $G(i)$ or $\alpha _i(g_{i+1})\ne 1$, we set $G^{\prime}(i)=G(i)$. Otherwise, since $G(i)$ is torsion free, there is an isomorphism $\iota\colon E_{G(i)}(g_{i+1})\to  \langle c_1\rangle $. Now we define $G^{\prime}(i)$ to be the corresponding HNN--extension
$$
G^{\prime}(i)=\langle G(i), t\; | \; e^t=\iota (e), \, e\in
E_{G(i)}(g_{i+1})\rangle .
$$
Then $G^{\prime}(i)$ is hyperbolic relative to $C$
and $\langle x, y\rangle$ is suitable in $G^{\prime}(i)$ by Lemma \ref{sgHNN}. Note also that  $G^{\prime}(i)$ is torsion free being an HNN-extension of a torsion free group.

{\bf Step 2.}
We now apply Theorem \ref{scq} to the group $G^{\prime}(i)$, the
subgroup $S=\langle x,y\rangle \le G^{\prime}(i)$, and the set of
elements $\{ t, c_{i+1}\} $ (or just $\{ c_{i+1}\}$ if $G^{\prime}(i)=G(i)$). Let $G(i+1)= \overline{G}$, where $\overline {G}$ is the quotient group provided by Theorem
\ref{scq}. Since $t$ becomes an element of $\langle x,y\rangle $
in $G(i+1)$, there is a naturally defined epimorphism $\phi
_{i+1}\colon G(0)\to G(i+1)$. Using Theorem \ref{scq} and the inductive assumption it is
straightforward to verify properties (a)--(e) for $G(i+1)$.

Observe that the group $G^{\prime}(i)$ admits an epimorphism $\beta _i$ to $\mathbb Z_2$ which sends the stable letter and $\Ker  (\alpha_i)$ to $1$. Indeed this follows immediately from the inductive assumption and our construction of $G^{\prime}(i)$. By Remark \ref{screm} and part (f) of the inductive assumption, the kernel of the natural epimorphism $G^{\prime}(i)\to G(i+1)$ is contained in $\Ker  \beta _i$. Hence $\beta _i$ induces an epimorphism $\alpha _{i+1}\colon G(i+1)\to \mathbb Z_2$.  Obviously (f) and (g) hold for $G(i+1)$.

Let $G=G(0)/M$,
where $M=\bigcup_{i=1}^\infty \Ker \phi_i$. By (d) $G$ is torsion free. It is also easy to see that $G$ is
$2$--generated. Indeed, $G(0)$ is generated by $x,y, c_1, c_2,
\ldots $ and hence condition (e) implies that $G$ is generated by $x$ and $y$.

Further notice that $M\le \Ker \alpha _0$ by (f). Let $H=(\Ker \alpha _0)/M$. Then $G/H$ is isomorphic to $G(0)/\Ker\alpha_0$, so $|G/H|=2$. Let $h$ be a nontrivial element of $H$. We take an arbitrary preimage $g\in G(0)$ of $h$. Observe that $\alpha _i(g)=1$ for every $i$ by (f). Hence (the image of) the element $g$ becomes conjugate to $c_1$ by an element $\Ker \alpha _i$ at a certain step according to (g). Therefore, all non-trivial elements of $H$ are conjugate in $H$.

Finally let $F$ and $a$ be the free subgroup and the element of $C$ provided by Lemma \ref{NCZ}, respectively. By part (c) of Lemma \ref{NCZ}, parts (a), (b), (d) of the inductive assumption, and Lemma \ref{malnorm}, for any two distinct elements $f_1,f_2\in F$, $af_1$ and $af_2$ are not conjugate in $G(i)$. Hence the same holds true in $G$. Since the natural map from $C$ to $G$ is injective by (a) and the (ordinary) growth function of $F$ is exponential, the conjugacy growth function of $G$ is exponential as well.
\end{proof}

\end{document}

%% file: fig1.pdf_tex

\begingroup
  \makeatletter
  \providecommand\color[2][]{%
    \errmessage{(Inkscape) Color is used for the text in Inkscape, but the package 'color.sty' is not loaded}
    \renewcommand\color[2][]{}%
  }
  \providecommand\transparent[1]{%
    \errmessage{(Inkscape) Transparency is used (non-zero) for the text in Inkscape, but the package 'transparent.sty' is not loaded}
    \renewcommand\transparent[1]{}%
  }
  \providecommand\rotatebox[2]{#2}
  \ifx\svgwidth\undefined
    \setlength{\unitlength}{272.70962526pt}
  \else
    \setlength{\unitlength}{\svgwidth}
  \fi
  \global\let\svgwidth\undefined
  \makeatother
  \begin{picture}(1,0.51541729)%
    \put(0,0){\includegraphics[width=\unitlength]{fig1.pdf}}%
    \put(0.48612047,0.00559855){\makebox(0,0)[lb]{\smash{$p$}}}%
    \put(0.47301919,0.27592872){\color[rgb]{0,0,0}\makebox(0,0)[lb]{\smash{$\Pi $}}}%
    \put(0.38931448,0.10365396){\color[rgb]{0,0,0}\makebox(0,0)[lb]{\smash{$\Gamma $}}}%
    \put(0.07164488,0.29499482){\color[rgb]{0,0,0}\makebox(0,0)[lb]{\smash{$\Delta $}}}%
    \put(0.67787492,0.1492753){\color[rgb]{0,0,0}\makebox(0,0)[lb]{\smash{$s_2$}}}%
    \put(0.25914412,0.15472275){\color[rgb]{0,0,0}\makebox(0,0)[lb]{\smash{$s_1$}}}%
    \put(0.48492937,0.06483963){\color[rgb]{0,0,0}\makebox(0,0)[lb]{\smash{$q_2$}}}%
    \put(0.47710099,0.14655389){\color[rgb]{0,0,0}\makebox(0,0)[lb]{\smash{$q_1$}}}%
  \end{picture}%
\endgroup

%% file: fig2.pdf_tex

\begingroup
  \makeatletter
  \providecommand\color[2][]{%
    \errmessage{(Inkscape) Color is used for the text in Inkscape, but the package 'color.sty' is not loaded}
    \renewcommand\color[2][]{}%
  }
  \providecommand\transparent[1]{%
    \errmessage{(Inkscape) Transparency is used (non-zero) for the text in Inkscape, but the package 'transparent.sty' is not loaded}
    \renewcommand\transparent[1]{}%
  }
  \providecommand\rotatebox[2]{#2}
  \ifx\svgwidth\undefined
    \setlength{\unitlength}{297.44035034pt}
  \else
    \setlength{\unitlength}{\svgwidth}
  \fi
  \global\let\svgwidth\undefined
  \makeatother
  \begin{picture}(1,0.41586575)%
    \put(0,0){\includegraphics[width=\unitlength]{fig2.pdf}}%
    \put(0.37396475,0.25581906){\color[rgb]{0,0,0}\rotatebox{22.8794861}{\makebox(0,0)[lb]{\smash{$\le K+\e $}}}}%
    \put(0.08728602,0.39326668){\color[rgb]{0,0,0}\makebox(0,0)[lb]{\smash{$t_2$}}}%
    \put(0.31271946,0.38998367){\color[rgb]{0,0,0}\makebox(0,0)[lb]{\smash{$r_2$}}}%
    \put(0.53924733,0.38670066){\color[rgb]{0,0,0}\makebox(0,0)[lb]{\smash{$t_1$}}}%
    \put(0.72528458,0.39107799){\color[rgb]{0,0,0}\makebox(0,0)[lb]{\smash{$r_1$}}}%
    \put(-0.00135531,0.21817272){\color[rgb]{0,0,0}\makebox(0,0)[lb]{\smash{$s_1$}}}%
    \put(0.04898421,0.00696568){\color[rgb]{0,0,0}\makebox(0,0)[lb]{\smash{$v_2$}}}%
    \put(0.28864405,0.01790903){\color[rgb]{0,0,0}\makebox(0,0)[lb]{\smash{$u_2$}}}%
    \put(0.53705857,0.01790903){\color[rgb]{0,0,0}\makebox(0,0)[lb]{\smash{$v_1$}}}%
    \put(0.72200161,0.0102487){\color[rgb]{0,0,0}\makebox(0,0)[lb]{\smash{$u_1$}}}%
    \put(0.92326718,0.17111624){\color[rgb]{0,0,0}\makebox(0,0)[lb]{\smash{$s_2$}}}%
  \end{picture}%
\endgroup

%% file: fig3.pdf_tex

\begingroup
  \makeatletter
  \providecommand\color[2][]{%
    \errmessage{(Inkscape) Color is used for the text in Inkscape, but the package 'color.sty' is not loaded}
    \renewcommand\color[2][]{}%
  }
  \providecommand\transparent[1]{%
    \errmessage{(Inkscape) Transparency is used (non-zero) for the text in Inkscape, but the package 'transparent.sty' is not loaded}
    \renewcommand\transparent[1]{}%
  }
  \providecommand\rotatebox[2]{#2}
  \ifx\svgwidth\undefined
    \setlength{\unitlength}{371.3485416pt}
  \else
    \setlength{\unitlength}{\svgwidth}
  \fi
  \global\let\svgwidth\undefined
  \makeatother
  \begin{picture}(1,0.5506317)%
    \put(0,0){\includegraphics[width=\unitlength]{fig3.pdf}}%
    \put(0.42002625,0.34659482){\color[rgb]{0,0,0}\makebox(0,0)[lb]{\smash{$r_1$}}}%
    \put(0.51895352,0.19641603){\color[rgb]{0,0,0}\makebox(0,0)[lb]{\smash{$r_2$}}}%
    \put(0.06603341,0.39546252){\color[rgb]{0,0,0}\makebox(0,0)[lb]{\smash{$t_1$}}}%
    \put(0.38307751,0.52418718){\color[rgb]{0,0,0}\makebox(0,0)[lb]{\smash{$p_1$}}}%
    \put(0.72038384,0.36924082){\color[rgb]{0,0,0}\makebox(0,0)[lb]{\smash{$s_1$}}}%
    \put(0.22217167,0.13086176){\color[rgb]{0,0,0}\makebox(0,0)[lb]{\smash{$t_2$}}}%
    \put(0.60953751,0.03789401){\color[rgb]{0,0,0}\makebox(0,0)[lb]{\smash{$p_2$}}}%
    \put(0.89916813,0.16423487){\color[rgb]{0,0,0}\makebox(0,0)[lb]{\smash{$s_2$}}}%
    \put(0.48915615,0.27984871){\color[rgb]{0,0,0}\makebox(0,0)[lb]{\smash{$r^\prime$}}}%
    \put(0.2126365,0.52895476){\color[rgb]{0,0,0}\makebox(0,0)[lb]{\smash{$x_1$}}}%
    \put(0.72155483,0.03407941){\color[rgb]{0,0,0}\makebox(0,0)[lb]{\smash{
$y_2$}}}%
    \put(0.46015667,0.42945734){\color[rgb]{0,0,0}\rotatebox{18.27745161}{\makebox(0,0)[lb]{\smash{$\le 2(K+\e)$}}}}%
    \put(0.27699886,0.02955066){\color[rgb]{0,0,0}\makebox(0,0)[lb]{\smash{$y_1$}}}%
    \put(0.65006199,0.53253045){\color[rgb]{0,0,0}\makebox(0,0)[lb]{\smash{$x_2$}}}%
  \end{picture}%
\endgroup

%% file: fig4.pdf_tex

\begingroup
  \makeatletter
  \providecommand\color[2][]{%
    \errmessage{(Inkscape) Color is used for the text in Inkscape, but the package 'color.sty' is not loaded}
    \renewcommand\color[2][]{}%
  }
  \providecommand\transparent[1]{%
    \errmessage{(Inkscape) Transparency is used (non-zero) for the text in Inkscape, but the package 'transparent.sty' is not loaded}
    \renewcommand\transparent[1]{}%
  }
  \providecommand\rotatebox[2]{#2}
  \ifx\svgwidth\undefined
    \setlength{\unitlength}{365.34154053pt}
  \else
    \setlength{\unitlength}{\svgwidth}
  \fi
  \global\let\svgwidth\undefined
  \makeatother
  \begin{picture}(1,0.57698834)%
    \put(0,0){\includegraphics[width=\unitlength]{fig4.pdf}}%
    \put(-0.00110342,0.26641792){\color[rgb]{0,0,0}\makebox(0,0)[lb]{\smash{$p_4$}}}%
    \put(0.93460315,0.26172354){\color[rgb]{0,0,0}\makebox(0,0)[lb]{\smash{$p_2$}}}%
    \put(0.28120836,0.48491873){\color[rgb]{0,0,0}\makebox(0,0)[lb]{\smash{$s_1$}}}%
    \put(0.74919217,0.45176028){\color[rgb]{0,0,0}\makebox(0,0)[lb]{\smash{$t_1$}}}%
    \put(0.51774504,0.40220999){\color[rgb]{0,0,0}\makebox(0,0)[lb]{\smash{$q_1$}}}%
    \put(0.44949248,0.29049815){\color[rgb]{0,0,0}\makebox(0,0)[lb]{\smash{$\Pi$}}}%
    \put(0.42725555,0.47889693){\color[rgb]{0,0,0}\makebox(0,0)[lb]{\smash{$\Gamma _1$}}}%
    \put(0.08667741,0.45081186){\color[rgb]{0,0,0}\makebox(0,0)[lb]{\smash{$\Delta $}}}%
    \put(0.17176501,0.55858946){\color[rgb]{0,0,0}\makebox(0,0)[lb]{\smash{$u_1$}}}%
    \put(0.51353345,0.55717133){\color[rgb]{0,0,0}\makebox(0,0)[lb]{\smash{$r_1$}}}%
    \put(0.88650068,0.55858946){\color[rgb]{0,0,0}\makebox(0,0)[lb]{\smash{$v_1$}}}%
    \put(0.39757903,0.17800863){\color[rgb]{0,0,0}\makebox(0,0)[lb]{\smash{$z$}}}%
    \put(0.43953945,0.00567106){\color[rgb]{0,0,0}\makebox(0,0)[lb]{\smash{$p_3$}}}%
  \end{picture}%
\endgroup

%% file: fig5.pdf_tex

\begingroup
  \makeatletter
  \providecommand\color[2][]{%
    \errmessage{(Inkscape) Color is used for the text in Inkscape, but the package 'color.sty' is not loaded}
    \renewcommand\color[2][]{}%
  }
  \providecommand\transparent[1]{%
    \errmessage{(Inkscape) Transparency is used (non-zero) for the text in Inkscape, but the package 'transparent.sty' is not loaded}
    \renewcommand\transparent[1]{}%
  }
  \providecommand\rotatebox[2]{#2}
  \ifx\svgwidth\undefined
    \setlength{\unitlength}{395.44613037pt}
  \else
    \setlength{\unitlength}{\svgwidth}
  \fi
  \global\let\svgwidth\undefined
  \makeatother
  \begin{picture}(1,1.00791789)%
    \put(0,0){\includegraphics[width=\unitlength]{fig5.pdf}}%
    \put(-0.00101942,0.72099062){\color[rgb]{0,0,0}\makebox(0,0)[lb]{\smash{$p_4$}}}%
    \put(0.86483806,0.71665362){\color[rgb]{0,0,0}\makebox(0,0)[lb]{\smash{$p_2$}}}%
    \put(0.14954752,0.59194621){\color[rgb]{0,0,0}\makebox(0,0)[lb]{\smash{$t_2$}}}%
    \put(0.6056025,0.58082299){\color[rgb]{0,0,0}\makebox(0,0)[lb]{\smash{$s_2$}}}%
    \put(0.25980049,0.92285737){\color[rgb]{0,0,0}\makebox(0,0)[lb]{\smash{$s_1$}}}%
    \put(0.69215755,0.89222321){\color[rgb]{0,0,0}\makebox(0,0)[lb]{\smash{$t_1$}}}%
    \put(0.18127512,0.81292527){\color[rgb]{0,0,0}\makebox(0,0)[lb]{\smash{$z_1$}}}%
    \put(0.6990985,0.71467751){\color[rgb]{0,0,0}\makebox(0,0)[lb]{\smash{$z_2$}}}%
    \put(0.35927686,0.64532617){\color[rgb]{0,0,0}\makebox(0,0)[lb]{\smash{$q_2$}}}%
    \put(0.47833006,0.8464451){\color[rgb]{0,0,0}\makebox(0,0)[lb]{\smash{$q_1$}}}%
    \put(0.00013641,0.25164727){\color[rgb]{0,0,0}\makebox(0,0)[lb]{\smash{$p_4^\prime$}}}%
    \put(0.86599392,0.2473103){\color[rgb]{0,0,0}\makebox(0,0)[lb]{\smash{$p_2^\prime$}}}%
    \put(0.15070335,0.12260289){\color[rgb]{0,0,0}\makebox(0,0)[lb]{\smash{$t_2^\prime$}}}%
    \put(0.6067583,0.11147967){\color[rgb]{0,0,0}\makebox(0,0)[lb]{\smash{$s_2^\prime$}}}%
    \put(0.26095635,0.45351406){\color[rgb]{0,0,0}\makebox(0,0)[lb]{\smash{$s_1^\prime$}}}%
    \put(0.6933134,0.4228799){\color[rgb]{0,0,0}\makebox(0,0)[lb]{\smash{$t_1^\prime$}}}%
    \put(0.18243095,0.34358198){\color[rgb]{0,0,0}\makebox(0,0)[lb]{\smash{$z_1^\prime$}}}%
    \put(0.70025436,0.24533419){\color[rgb]{0,0,0}\makebox(0,0)[lb]{\smash{$z_2^\prime$}}}%
    \put(0.36043272,0.17598288){\color[rgb]{0,0,0}\makebox(0,0)[lb]{\smash{$q_2^\prime$}}}%
    \put(0.47948586,0.37710177){\color[rgb]{0,0,0}\makebox(0,0)[lb]{\smash{$q_1^\prime$}}}%
    \put(0.04238877,0.53984598){\color[rgb]{0,0,0}\makebox(0,0)[lb]{\smash{$v_2$}}}%
    \put(0.3136645,0.56256922){\color[rgb]{0,0,0}\makebox(0,0)[lb]{\smash{$r_2$}}}%
    \put(0.70993056,0.54506124){\color[rgb]{0,0,0}\makebox(0,0)[lb]{\smash{$u_2$}}}%
    \put(0.14799595,0.453661){\color[rgb]{0,0,0}\makebox(0,0)[lb]{\smash{$u_1^\prime$}}}%
    \put(0.50678104,0.43378747){\color[rgb]{0,0,0}\makebox(0,0)[lb]{\smash{$r_1^\prime$}}}%
    \put(0.81358204,0.45626857){\color[rgb]{0,0,0}\makebox(0,0)[lb]{\smash{$v_1^\prime$}}}%
    \put(0.40379381,0.51751386){\color[rgb]{0,0,0}\makebox(0,0)[lb]{\smash{$j$}}}%
    \put(0.43481173,0.57879212){\color[rgb]{0,0,0}\makebox(0,0)[lb]{\smash{$\Gamma _3$}}}%
    \put(0.41527344,0.74323767){\color[rgb]{0,0,0}\makebox(0,0)[lb]{\smash{$\Pi$}}}%
    \put(0.39472937,0.917294){\color[rgb]{0,0,0}\makebox(0,0)[lb]{\smash{$\Gamma _1$}}}%
    \put(0.08007882,0.89134699){\color[rgb]{0,0,0}\makebox(0,0)[lb]{\smash{$\Delta $}}}%
    \put(0.15868885,0.99091968){\color[rgb]{0,0,0}\makebox(0,0)[lb]{\smash{$u_1$}}}%
    \put(0.47443908,0.98960951){\color[rgb]{0,0,0}\makebox(0,0)[lb]{\smash{$r_1$}}}%
    \put(0.81901301,0.99091968){\color[rgb]{0,0,0}\makebox(0,0)[lb]{\smash{$v_1$}}}%
    \put(0.032769,0.00658362){\color[rgb]{0,0,0}\makebox(0,0)[lb]{\smash{$v_2^\prime$}}}%
    \put(0.29893443,0.00523934){\color[rgb]{0,0,0}\makebox(0,0)[lb]{\smash{$r_2^\prime $}}}%
    \put(0.68608413,0.00523934){\color[rgb]{0,0,0}\makebox(0,0)[lb]{\smash{$u_2^\prime$}}}%
    \put(0.07933549,0.39767153){\color[rgb]{0,0,0}\makebox(0,0)[lb]{\smash{$\Delta ^\prime$}}}%
  \end{picture}%
\endgroup